\documentclass{amsart}
\usepackage[mathscr]{eucal}
\usepackage{amssymb}
\usepackage{amsmath}
\usepackage{cases}
\usepackage{latexsym}
\usepackage[dvips]{graphicx}

\makeatletter
\def\mojiparline#1{
    \newcounter{mpl}
    \setcounter{mpl}{#1}
    \@tempdima=\linewidth
    \advance\@tempdima by-\value{mpl}zw
    \addtocounter{mpl}{-1}
    \divide\@tempdima by \value{mpl}
    \advance\kanjiskip by\@tempdima
    \advance\parindent by\@tempdima
}
\makeatother

\def\R{\mathbb{R}}
\def\N{\mathbb{N}}
\def\calB{\mathcal{B}}
\def\calC{\mathcal{C}}
\def\calF{\mathcal{F}}
\def\calL{\mathcal{L}}
\def\calM{\mathcal{M}}
\def\calN{\mathcal{N}}

\def\calR{\mathcal{R}}
\def\calS{\mathcal{S}}
\def\calT{\mathcal{T}}
\def\calZ{\mathcal{Z}}
\def\bbB{\mathbb{B}}
\def\bbX{\mathbb{X}}
\def\bbY{\mathbb{Y}}
\def\e{\varepsilon}

\theoremstyle{plain}
	\newtheorem{theorem}{Theorem}[section]
	\newtheorem{lemma}[theorem]{Lemma}
	\newtheorem{corollary}[theorem]{Corollary}
	
	\newtheorem{proposition}[theorem]{Proposition}
	\newtheorem{remark}[theorem]{Remark}
	
\theoremstyle{plain}
	\newtheorem{maintheorem}{Theorem}

\makeatletter
 \@addtoreset{equation}{section}
\makeatother

\begin{document}



\title[Supercritical Neumann problem]{Structure of\\ the positive radial solutions for\\ the supercritical Neumann problem $\e^2\Delta u-u+u^p=0$ in a ball}


\author{Yasuhito Miyamoto}

\thanks{This work was partially supported by the Japan Society for the Promotion of Science, Grant-in-Aid for Young Scientists (B) (Subject No.~24740100) and by Keio Gijuku Academic Development Funds.\\
To appear in {\it The special issue for the 20th anniversary, the Journal of Mathematical Sciences, the University of Tokyo.}}

\address{Graduate School of Mathematical Sciences, The University of Tokyo,
3-8-1 Komaba, Meguro-ku, Tokyo 153-8914, Japan}
\email{miyamoto@ms.u-tokyo.ac.jp}
\subjclass[2000]{Primary 35J25, 25B32; Secondary 34C23, 34C10.}
\date{\today}

\begin{abstract}
We are interested in the structure of the positive radial solutions of the supercritical Neumann problem in a unit ball
\[
\begin{cases}
\e^2\left(U''+\frac{N-1}{r}U'\right)-U+U^p=0, & 0<r<1,\\
U'(1)=0, & \\
U>0, & 0<r<1,
\end{cases}
\]
where $N$ is the spatial dimension and $p>p_S:=(N+2)/(N-2)$, $N\ge 3$.
We show that there exists a sequence $\{\e_n^*\}_{n=1}^{\infty}$ ($\e_1^*>\e_2^*>\cdots\rightarrow 0$) such that this problem has infinitely many singular solutions $\{(\e_n^*,U_n^*)\}_{n=1}^{\infty}\subset\R\times (C^2(0,1)\cap C^1(0,1])$ and that the nonconstant regular solutions consist of infinitely many smooth curves in the $(\e,U(0))$-plane.
It is shown that each curve blows up at $\e_n^*$ and if $p_S<p<p_{JL}$, then each curve has infinitely many turning points around $\e_n^*$. Here,
\[
p_{JL}:=
\begin{cases}
1+\frac{4}{N-4-2\sqrt{N-1}} & (N\ge 11),\\
\infty & (2\le N\le 10).
\end{cases}
\]
In particular, the problem has infinitely many solutions if $\e\in\{\e_n^*\}_{n=1}^{\infty}$.
We also show that there exists $\bar{\e}>0$ such that the problem has no nonconstant regular solution if $\e>\bar{\e}$.
The main technical tool is the intersection number between the regular and singular solutions.
\end{abstract}




\maketitle
\tableofcontents


\section{Introduction and main results}
Let $\Omega\subset\R^N$, $N\ge 3$ be a bounded domain with smooth boundary. We are concerned with the elliptic Neumann problem
\begin{equation}\label{S1E1}
\begin{cases}
\e^2\Delta U-U+U^p=0, & \textrm{in}\ \Omega,\\
\frac{\partial}{\partial\nu}U=0, & \textrm{on}\ \partial\Omega,\\
U>0, & \textrm{in}\ \Omega,
\end{cases}
\end{equation}
where $p>1$ and $\e\in\R_+:=\{x;\ x>0\}$. The problem (\ref{S1E1}) arises in physical and biological models. In particular, (\ref{S1E1}) appears in the stationary problem of the Keller-Segel model for chemotaxis aggregation \cite{KS70} and the shadow system of the Gierer-Meinhardt model for biological pattern formations \cite{GM72}. For these two decades (\ref{S1E1}) has attracted considerable attention and solutions with various shapes have been found. See \cite{W97,W99} for single-peak solutions, \cite{GW00} for multi-peak solutions, and \cite{MM02} for boundary concentrating solutions. Many papers, including \cite{GW00,MM02,NT91,W97,W99}, study the subcritical case $1<p<p_S$, where
\[
p_S:=
\begin{cases}
\frac{N+2}{N-2}, & (N\ge 3),\\
\infty, & (N=1,2).
\end{cases}
\]
In this case the compact embedding $H^1(\Omega)\hookrightarrow L^{p+1}(\Omega)$ is applicable, hence a variational method works well.
In this paper we are interested in the solution structure of (\ref{S1E1}) in the supercritical case $p>p_S$.
Specifically, we consider the case where the domain $\Omega$ is the unit ball $B:=\{x\in\R^N;\ |x|<1\}$, and study the bifurcation diagram of the positive radial solutions in the $(\lambda,U)$ plane.
The problem (\ref{S1E1}) can be reduced to the ODE
\begin{equation}\label{S1E2}
\begin{cases}
\e^2\left( U''+\frac{N-1}{r}U'\right)-U+U^p=0, & 0<r<1,\\
U'(1)=0, & \\
U>0,& 0\le r\le 1,
\end{cases}
\end{equation}
Throughout the present paper we define $f(U):=-U+U^p$ and $\lambda:=1/\e^2$. 
Then $\lambda\in\R_+$ and $\lambda$ diverges as $\e\downarrow 0$. Since we study the bifurcation diagram of (\ref{S1E2}), it is convenient to transform (\ref{S1E2}) into
\begin{equation}\label{N}
\begin{cases}
U''+\frac{N-1}{r}U'+\lambda f(U)=0, & 0<r<1,\\
U'(1)=0, & {}\\
U>0,& 0\le r\le 1.
\end{cases}
\end{equation}
We need some notations.
We call the constant solutions $\{(\lambda,1)\}$ {\it the trivial branch} which is denoted by $\calC_0$.
Let $\Delta_N$ denote the Neumann Laplacian, and let $\{\mu_n\}_{n=0}^{\infty}$ be the eigenvalues of $-\Delta_N$ on $B$ in the space of radial functions.
Since each $\mu_n$ is simple, $0=\mu_0<\mu_1<\mu_2<\cdots$.
Let $\bar{\lambda}_n:=\mu_n/(p-1)$.
By Crandall-Rabinowitz bifurcation theorem we easily see that $(\bar{\lambda}_n,1)$ $(n=1,2,\cdots)$ is a bifurcation point from which a curve of nontrivial solutions emanates.
We denote the closure of the curve by $\calC_n$.
We will see in Proposition~\ref{LPR} of the present paper that each curve $\calC_n$, which we call the branch, can be locally parametrized by $\gamma:=U_n(0)$.
Then each curve $\calC_n$ can be described as $\{(\lambda_n(\gamma),U_n(r,\gamma))\}$ ($\gamma:=U_n(0,\gamma)$).
We define
\[
\calC^+_n:=\calC_n\cap\{\gamma>1\},\quad\calC^-_n:=\calC_n\cap\{\gamma<1\}.
\]
By $\calZ_I[v(\,\cdot\,)]$ we denote the number of the zeros of the function $v(\,\cdot\,)$ in the interval $I\subset\R$, i.e.,
\[
\calZ_I[v(\,\cdot\,)]:=\sharp\{x\in I;\ v(x)=0\}.
\]
Let $(\lambda_n,U_n)\in\calC_n$.
Every zero of $U_n(r)-1$ $(0\le r\le 1)$ is simple, because of the uniqueness of the solution of the ODE (\ref{N}).
If $(\lambda_n,U_n)$ is near $(\bar{\lambda}_n,1)$, then $U_n(r)-1$ is close to the $n$-th eigenfunction of the corresponding eigenvalue problem. By Sturm-Liouville theory, $\calZ_{[0,1]}[U_n(\,\cdot\,,\gamma)-1]=n$ provided that $(\lambda_n,U_n)$ is near $(\bar{\lambda}_n,1)$.
Since each zero is simple and $U_n(r)-1$ continuously changes along each $\calC_n$, $\calZ_{[0,1]}[U_n(\,\cdot\,,\gamma)-1]$ is preserved along each $\calC_n$.
Therefore, $\calZ_{[0,1]}[U_n(\,\cdot\,,\gamma)-1]=n$ if $U_n\in \calC_n\backslash\calC_0$.

Let us recall known results about (\ref{N}).
It was shown in \cite{K89} that in the subcritical case $1<p<p_S$ the set of the regular solutions of (\ref{N}) is bounded in $L^{\infty}$, i.e.,
\[
(\gamma^*:=)\sup\{\|U\|_{\infty};\ \textrm{$U$ is a solution of (\ref{N}).}\}<\infty.
\]
and that $\lambda_1(\gamma)\rightarrow\infty$ as $\gamma\uparrow\gamma^*$.
Since $\left.\frac{d\lambda_1(\gamma)}{d\gamma}\right|_{\gamma=1}<0$ (\cite{MY13}), the branch $\calC_1$ has at least one turning point.
Moreover, if $\gamma(<\gamma^*)$ is close to $\gamma^*$, then $U_n(\gamma)$ is nondegenerate, hence $\frac{d\lambda_1(\gamma)}{d\gamma}>0$ (\cite{W99}).
In the critical case $p=p_S$ the solution structure depends on the spatial dimension $N$.
For $N\in\{4,5,6\}$, (\ref{N}) admits a (nonconstant) radially decreasing solution for $0<\lambda<\bar{\lambda}_1$ (\cite{AY91}), which implies that $\calC_1$ is unbounded in $\gamma$.
For $N\ge 7$, there is $\underline{\lambda}>0$ such that (\ref{N}) has no nonconstant solution for $0<\lambda<\underline{\lambda}$ (\cite{AY91,AY97,BKP91}).
When $N=3$, the structure depends on the radius of the ball.
See \cite{AY91} for a partial result.
In the supercritical case $p>p_S$, there are few results about (\ref{S1E1}).
When $p=p_S+\e$ for small $\e>0$, \cite{DMP05} constructed a bubble tower solution of (\ref{S1E1}).
See \cite{N83,LN88,GN12} for other results.
A brief history of this problem is written in \cite{GN12}.

On the other hand, the branch of the positive solutions of the critical or supercritical Dirichlet problem
\begin{equation}\label{S1E4}
\begin{cases}
U''+\frac{N-1}{r}U'+\lambda g(U)=0, & 0<r<1,\\
U(1)=0, & \\
U>0, & 0\le r<1,
\end{cases}
\end{equation}
was studied by \cite{BN87,DF07,GW11,JL73,Mi13}.
In \cite{JL73} the case $g(U)=(1+U)^p$ was studied.
The structure depends on $p$ and $N$.
Let
\[
p_{JL}:=
\begin{cases}
1+\frac{4}{N-4-2\sqrt{N-1}}, & (N\ge 11),\\
\infty, & (2\le N\le 10).
\end{cases}
\]
When $p_S<p<p_{JL}$, Joseph-Lundgren\cite{JL73} have shown that the branch which emanates from $(0,0)$ has infinitely many turning points around $\lambda=\lambda^*:=\theta\left( N-2-\theta\right)$ and blows up at $\lambda=\lambda^*$, where
\begin{equation}\label{theta}
\theta:=\frac{2}{p-1}.
\end{equation}
Moreover, there is a singular solution $U^*=r^{-\theta}-1\in H^1_0(B)$ for $\lambda=\lambda^*$.
When $p\ge p_{JL}$, the branch exists for $0<\lambda<\lambda^*$, does not have a turning point, and blows up at $\lambda=\lambda^*$.
Moreover the singular solution $(\lambda^*,U^*)$ exists.
In particular, (\ref{N}) has a unique regular solution for $0<\lambda<\lambda^*$.

First we study the problem
\begin{equation}\label{S3E2}
\begin{cases}
u''+\frac{N-1}{s}u'+f(u)=0, & s>0,\\
u(0)=\gamma,\ u'(0)=0. &
\end{cases}
\end{equation}
In the study of the bifurcation diagram singular solutions play an important role.
\begin{theorem}\label{Th0}
Suppose that $p>p_S$.
There is a singular positive solution $u^*(s)$ of
\[
\begin{cases}
u''+\frac{N-1}{s}u'+f(u)=0,& 0<s<\infty,\\
u(s)\rightarrow \infty\quad (s\downarrow 0).
\end{cases}
\]
$u^*(s)$ oscillates around $1$ hence it has infinitely many critical points $\{s_n^*\}_{n=1}^{\infty}$. Moreover, $u(s,\gamma)\rightarrow u^*(s)$ in $C_{loc}^0(0,\infty)$ as $\gamma\rightarrow\infty$.
Here $u(s,\gamma)$ is the solution of (\ref{S3E2}).
\end{theorem}
When $p>p_{JL}$, this theorem was already proved by \cite[Theorem~1.1]{CCCT10}.
Our method is different from \cite{CCCT10} and it works in the case $p_S<p\le p_{JL}$.
The same oscillation property of the regular solution of (\ref{S3E2}) was obtained by \cite{N83}.
Using almost the same argument, one can see that $u^*(s^*_1)<u^*(s^*_3)<\cdots<1$, $u^*(s^*_2)>u^*(s^*_4)>\cdots>1$, and $s^*_{n+1}-s^*_n\rightarrow \pi/\sqrt{p-1}$ $(n\rightarrow\infty)$.

By scaling the singular solution $u^*$ we construct singular solutions of (\ref{N}).
The first main result is the following:
\begin{maintheorem}\label{Th1}
Suppose that $p>p_S$.
The problem (\ref{N}) has infinitely many singular solutions $(\lambda_n^*,U_n^*(r))\in\R_+\times (C^2(0,1)\cap C^0(0,1]\cap H^1(B))$ ($n=1,2,\cdots$ and $\lambda_1^*<\lambda_2^*<\cdots\rightarrow \infty$) such that the following assertions hold:\\
(i) $U_n^*(r)$ satisfies
\[
U_n^*(r)=A(p,N)(\sqrt{\lambda_n^*}r)^{-\theta}(1+o(1))\quad\textrm{as}\quad (r\downarrow 0),
\]
where
\begin{equation}\label{Th2E2}
A(p,N):=\left\{\theta (N-2-\theta)\right\}^{\frac{1}{p-1}}.
\end{equation}
(ii) $\calZ_{(0,1]}[U_n^*(\,\cdot\,)-1]=n$.\\
(iii) $U_n^*(r)>0$ $(0<r\le 1)$.\\
Moreover, the singular solution $(\lambda_n^*,U_n^*)$ is unique, i.e., if $(\tilde{\lambda}^*_n,\tilde{U}^*_n)$ is a singular solution such that (i) and (ii) hold, then $(\tilde{\lambda}^*_n,\tilde{U}^*_n)=(\lambda_n^*,U_n^*)$.
\end{maintheorem}
In the proof of Theorem~\ref{Th1} we see that $\sqrt{\lambda^*_n}=s^*_n$.
Hence, $\sqrt{\lambda^*_{n+1}}-\sqrt{\lambda^*_n}\rightarrow\pi/\sqrt{p-1}$ $(n\rightarrow\infty)$.

The second main result of the paper is the following:
\begin{maintheorem}\label{Th2}
Suppose that $p>p_S$.
Let $\lambda_n^*$ $(n=1,2,\cdots)$ be as in Theorem~\ref{Th1}.
Let $\calS$ denote the set of the regular solutions of (\ref{N}).
Then
\begin{equation}\label{Th3E0}
\calS=\calC_0\cup\bigcup_{n=1}^{\infty}\left(\calC^+_n\cup\calC^-_n\right).
\end{equation}
Each $\calC_n$ $(n=1,2,\cdots)$ can be parametrized by $\gamma=U_n(0)$, hence $\calC_n$ can be described as $\{(\lambda_n(\gamma),U_n(r,\gamma))\}$.
Moreover, $\lambda_n(\gamma)\in C^1(0,\infty)$ and the following assertions hold:\\
(i) For each $n\ge 1$, $\lambda_n(1)=\bar{\lambda}_n$.\\
(ii) For each $n\ge 1$, $\lambda_n(\gamma)\rightarrow\lambda_n^*$ $(\gamma\rightarrow\infty)$.\\
(iii) If $p_S<p<p_{JL}$, then for each $n\ge 1$, $\lambda_n(\gamma)$ oscillates around $\lambda_n^*$ infinitely many times as $\gamma\rightarrow\infty$.\\
(iv) For each $n\ge 1$, $\lambda_n(\gamma)\rightarrow\infty$ $(\gamma\downarrow 0)$.\\
(v) If $\gamma>0$ is small, then $U_1(r,\gamma)$ is nondegenerate, and it becomes a boundary concentrating solution as $\gamma\downarrow 0$.\\
(vi) For each $\gamma\in\R_+$, $\lambda_1(\gamma)<\lambda_2(\gamma)<\cdots$.
\end{maintheorem}
Figure~\ref{fig1} shows the bifurcation diagram in the case $p_S<p<p_{JL}$.
\begin{figure}
\begin{center}
\includegraphics[scale=0.7]{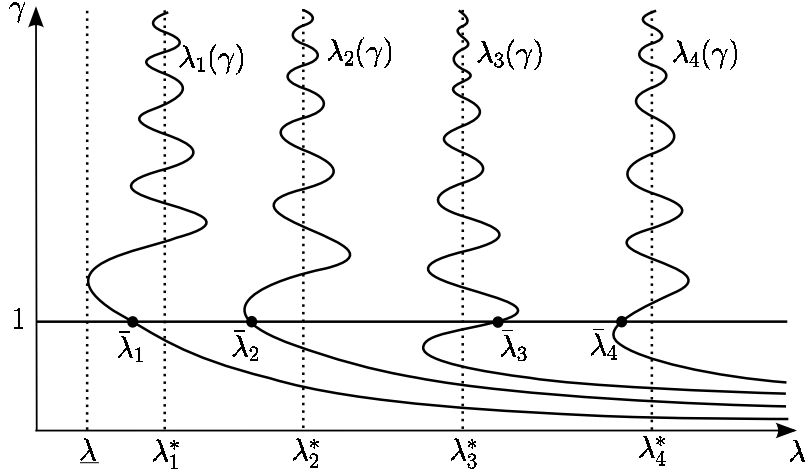}
\caption{The schematic picture of the bifurcation diagram of (\ref{N}) in the case $p_S<p<p_{JL}$.}
\label{fig1}
\end{center}
\end{figure}

(iv), (v), and (vi) of Theorem~\ref{Th2} hold for every $p>1$.

It follows from Theorem~\ref{Th2} (iii) that each branch $\calC^+_n$ has infinitely many turning points around $\lambda^*_n$ if $p_S<p<p_{JL}$.

\begin{remark}
If $3\le N\le 10$, then $p_{JL}=\infty$. Hence, the supercritical equation (\ref{N}) always has infinitely many solutions for $\lambda=\lambda^*_n$ $(n=1,2,\cdots)$ and each $\calC^+_n$ has infinitely many turning points.
\end{remark}

\begin{corollary}\label{Cor1}
Suppose that $p>p_S$.
There exists $\underline{\lambda}>0$ such that 
if $0<\lambda<\underline{\lambda}$, then (\ref{N}) has no nonconstant solution.
Moreover, there is $\bar{\lambda}>0$ such that the radially decreasing solution does not exist for $\lambda>\bar{\lambda}$.
\end{corollary}


Let us explain technical details.
In \cite{JL73} they used a special change of variables in the study of (\ref{S1E4}) with $g(U)=(1+U)^p$.
The equation (\ref{S1E4}) can be transformed to a first order autonomous system.
See (\ref{S21E2+}) in the present paper.
They used a phase plane analysis to obtain the bifurcation diagram.
However, for a general nonlinearity $g$, we cannot expect such a change of variables.
In \cite{Mi13} the author studied the bifurcation diagram of the solutions of (\ref{S1E4}) when $g(U)=U^p+h(U)(>0)$ and $h(U)$ is a lower order term.
Let $u(s):=U(r)$ and $s:=\sqrt{\lambda}r$. The equation (\ref{S1E4}) becomes
\begin{equation}\label{S1E6}
\begin{cases}
u''+\frac{N-1}{s}u'+g(u)=0, & 0<s<\sqrt{\lambda},\\
u(\sqrt{\lambda})=0,\\
u>0, & 0\le s<\sqrt{\lambda}.
\end{cases}
\end{equation}
The equation (\ref{S1E6}) has a singular solution $(\lambda^*,u^*(s))$.
Let $(\lambda,u(s,\gamma))$ denote the regular solution of (\ref{S1E6}) such that $u(0,\gamma)=\gamma$ and $u'(0,\gamma)=0$.
It was shown that $\lambda$ is a $C^1$-function of $\gamma$.
He used the intersection number between $u^*(s)$ and $u(s,\gamma)$ which can be written as
\[
\calZ_{I_{\gamma}}[u^*(\,\cdot\,)-u(\,\cdot\,,\gamma)],
\]
where $I_{\gamma}:=(0,\min\{\sqrt{\lambda^*},\sqrt{\lambda(\gamma)}\}]$. 
He showed by a scaling argument that if $p_S<p<p_{JL}$, then 
\begin{equation}\label{S1E6+}
\calZ_{I_{\gamma}}[u^*(\,\cdot\,)-u(\,\cdot\,,\gamma)]\rightarrow\infty\quad (\gamma\rightarrow\infty).
\end{equation}
Because of the uniqueness of the solution of the ODE (\ref{S1E6}), each zero of $u^*(\,\cdot\,)-u(\,\cdot\,,\gamma)$ is simple.
Hence, each zero depends continuously on $\gamma$.
The number of the zeros in $I_{\gamma}$ is preserved if another zero does not come from the boundary of $I_{\gamma}$.
Since $u^*(0)-u(0,\gamma)=\infty$, a zero cannot enter $I_{\gamma}$ from $s=0$.
Because of (\ref{S1E6+}), a zero enters $I_{\gamma}$ from $s=\min\{\sqrt{\lambda^*},\sqrt{\lambda(\gamma)}\}$ infinitely many times as $\gamma\rightarrow\infty$.
Therefore, $\lambda(\gamma)$ oscillates around $\lambda^*$ infinitely many times which indicates that the branch has infinitely many turning points.
This method cannot be directly applied to the Neumann problem, since in the Neumann case $u(\sqrt{\lambda(\gamma)},\gamma)$ is not necessarily equal to $u^*(\sqrt{\lambda^*})$ and the oscillation of $\lambda(\gamma)$ around $\lambda^*$ is not trivial.
The shape of the singular solution $u^*$ of the Neumann problem becomes important to study the behavior of each zero.

In our problem (\ref{N}) we see that $\calC_n$ is locally parametrized by $\gamma=U_n(0)$ (Proposition~\ref{LPR}).
We can write the solution as $(\lambda_n(\gamma),U_n(r,\gamma))$.

Section~3 is devoted to the study of the fundamental property of $\lambda_n(\gamma)$.
In Lemma~\ref{BDDB} we show that for each $\gamma^*>1$, $\lambda_n(\gamma)$ does not diverge as $\gamma\uparrow\gamma^*$.
Therefore, combining this result and the local parametrization of $\calC_n$, we show in Lemma~\ref{GPR} that the domain of $\lambda_n(\gamma)$ can be extended to $\gamma>1$.
It is perhaps interesting to note that the nondivergence of $\lambda_n(\gamma)$ is proved by the nonexistence of the entire positive solution of $\Delta u-u+u^p=0$ $(p\ge p_S)$ which is proved by the Pohozaev identity.
We also show that $\lambda_n(\gamma)$ can be extended to $0<\lambda<1$.
Hence, $\lambda_n(\gamma)$ is defined in $\gamma\in\R_+$.

In Section~4 we study the singular solution of (\ref{N}) and prove Theorem~\ref{Th1}. Let $u(s):=U(r)$ and $s:=\sqrt{\lambda}r$.
The equation (\ref{N}) is transformed to the problem
\begin{equation}\label{S1E7}
\begin{cases}
u''+\frac{N-1}{s}u'+f(u)=0, & 0<s<\sqrt{\lambda},\\
u'(\sqrt{\lambda})=0,\\
u>0, & 0\le s\le\sqrt{\lambda}.
\end{cases}
\end{equation}
In Lemma~\ref{S41L0} we construct the singular solution $u^*(s)$ of the equation in (\ref{S1E7}) near $s=0$ and show that $u^*(s)=As^{-\theta}(1+o(1))$ ($s\downarrow 0$).
Here $A:=A(p,N)$ and $A(p,N)$ is defined by (\ref{Th2E2}).
In Lemma~\ref{S42L1} we show that the domain of $u^*(s)$ can be extended to $0<s<\infty$, that $u^*(s)$ satisfies the equation in (\ref{S1E7}), and that $u^*(s)>0$ for $s>0$.
In Lemma~\ref{S43L1} we show that $u^*(s)$ oscillates around $1$ infinitely many times as $s\rightarrow\infty$ and that $u^*(s)$ has the set of the critical points $\{s^*_n\}_{n=1}^{\infty}$ of $u^*$ such that $0<s^*_1<s^*_2<\cdots\rightarrow \infty$ and
\[
\begin{cases}
\textrm{$s^*_n$ is a local minimum point of $u^*$ and $u^*(s^*_n)<1$ if $n\in\{1,3,5,\cdots\}$},\\
\textrm{$s^*_n$ is a local maximum point of $u^*$ and $u^*(s^*_n)>1$ if $n\in\{2,4,6,\cdots\}$}.\\
\end{cases}
\]
We set $\lambda^*_n:=(s^*_n)^2$ and $U^*_n(r):=u^*(s)$ $(s=\sqrt{\lambda_n^*}r)$.
Then, $(\lambda^*_n,U^*_n)$ is a singular solution stated in Theorem~\ref{Th1}.

Let $(\lambda_n(\gamma),u(s,\gamma))$ denote the solution of (\ref{S1E7}) such that $u(0,\gamma)=\gamma$ and $u_s(0,\gamma)=0$.
In Section~5 we show that $\lambda_n(\gamma)\rightarrow\lambda^*_n$ as $\gamma\rightarrow\infty$ and that $u(s,\gamma)$ converges to $u^*(s)$ in an appropriate sense. In \cite{MP91} Merle and Peletier proved a similar convergence result for the Dirichlet problem
\[
\begin{cases}
U''+\frac{N-1}{r}U'+\lambda U+U^p=0, & 0<r<1,\\
U(1)=0,\\
U>0, & 0\le r<1.
\end{cases}
\]
when $p>p_S$.
In Theorem~\ref{S5T1} we show that $u(s,\gamma)\rightarrow u^*(s)$, following arguments in the proof of \cite[Theorem A]{MP91}.

In Section~6 we show that $\lambda_n(\gamma)$ oscillates around $\lambda^*_n$ if $p_S<p<p_{JL}$.
Let $\rho:=\gamma^{\frac{p-1}{2}}s$. We define $\tilde{u}(\rho,\gamma):=u(s,\gamma)/\gamma$ and $\tilde{u}^*(\rho):=u^*(s)/\gamma$.
We use the intersection number between $\tilde{u}$ and $\tilde{u}^*$.
The function $\tilde{u}(\rho,\gamma)$ satisfies
\begin{equation}\label{S43L1E9}
\begin{cases}
\tilde{u}''+\frac{N-1}{\rho}\tilde{u}'+\tilde{u}^p-\frac{1}{\gamma^{p-1}}\tilde{u}=0, & 0<\rho<\infty,\\
\tilde{u}(0)=1,\ \tilde{u}'(0)=0. &
\end{cases}
\end{equation}
Let $\bar{u}(\rho,\gamma)$ be the regular solution of
\begin{equation}\label{S1E7+}
\begin{cases}
\bar{u}''+\frac{N-1}{\rho}\bar{u}'+\bar{u}^p=0,\quad 0<\rho<\infty,\\
\bar{u}(0)=\gamma,\ \bar{u}'(0)=0.
\end{cases}
\end{equation}
We show that as $\gamma\rightarrow\infty$,
\[
\tilde{u}(\rho,\gamma)\rightarrow \bar{u}(\rho,1)\quad\textrm{in}\quad C^2_{loc}(0,\infty)\cap C^0_{loc}[0,\infty)
\]
and
\[
\tilde{u}^*(\rho)\rightarrow \bar{u}^*(\rho)\quad\textrm{in}\quad C^0_{loc}(0,\infty),
\]
where $\bar{u}^*(\rho)$ is given by (\ref{singent}) which is a singular solution of the equation in (\ref{S1E7+}).
In Section~2 we recall the fact that $\calZ_{(0,\infty)}[\bar{u}^*(\,\cdot\,)-\bar{u}(\,\cdot\,,1)]=\infty$. Hence, for each $\delta>0$, 
\begin{equation}\label{S1E7++}
\calZ_{(0,\delta)}[u^*(\,\cdot\,)-u(\,\cdot\,,\gamma)]\rightarrow\infty\quad (\gamma\rightarrow\infty),
\end{equation}
since $s\in (0,\delta)$ is corresponding to $\rho\in (0,\delta\gamma^{\frac{N-1}{2}})$ and $\delta\gamma^{\frac{N-1}{2}}\rightarrow\infty$ ($\gamma\rightarrow\infty$).
Since each zero of $u^*(\,\cdot\,)-u(\,\cdot\,,\gamma)$ is simple, each zero depends continuously on $\gamma$.
The divergence (\ref{S1E7++}) tells us that a zero which is simple enters the interval $(0,\sqrt{\lambda^*_n}]$ from $s=\sqrt{\lambda^*_n}$ infinitely many times.
Therefore, there exists a sequence of large numbers $\{\gamma_j\}_{j=1}^{\infty}$ ($\gamma_1<\gamma_2<\cdots\rightarrow \infty$) such that $u^*(\sqrt{\lambda^*_n})=u(\sqrt{\lambda^*_n},\gamma_j)$ and the following holds:
$u_s(\sqrt{\lambda^*_n},\gamma_j)<0$ for $j\in\{1,3,5,\cdots\}$ and $u_s(\sqrt{\lambda^*_n},\gamma_j)>0$ for $j\in\{2,4,6,\cdots\}$.
In Theorem~\ref{S6T1}, using the convergence $u(s,\gamma)\rightarrow u^*(s)$, we show that if $n\in\{1,3,5,\cdots\}$ (resp. $n\in\{2,4,6,\cdots\}$)
\begin{equation}\label{S1E8}
\lambda_n(\gamma_j)
\begin{cases}
>\lambda^*_n, & (j\in\{1,3,5,\cdots\}),\\
<\lambda^*_n, & (j\in\{2,4,6,\cdots\}),
\end{cases}
\end{equation}
\[
\left(
\textrm{resp.}\ 
\lambda_n(\gamma_j)
\begin{cases}
<\lambda_n^*, & (j\in\{1,3,5,\cdots\}),\\
>\lambda_n^*, & (j\in\{2,4,6,\cdots\}).
\end{cases}
\right)
\]
which implies that $\lambda_n(\gamma)$ oscillates around $\lambda^*_n$ infinitely many times as $\gamma\rightarrow\infty$.

In Section~7 we construct a smooth branch of boundary concentrating solutions, using a standard blow-up argument with the contraction mapping theorem.
We also show that this branch is in $\calC_1$ and that $\lambda_1(\gamma)\rightarrow\infty$ $(\gamma\downarrow 0)$.

The paper consists of 8 sections.
In Section~2 we recall known result about the intersection number and the nonexistence of the entire solution.
In Section~3 we collect fundamental properties of $\lambda_n(\gamma)$.
In Section~4 we prove Theorem~\ref{Th1}.
In Section~5 we prove the convergence to $(\lambda_n^*,U_n^*)$.
In Section~6 we prove the oscillation of $\lambda_n(\gamma)$.
In Section~7 we prove (iv) and (v) of Theorem~\ref{Th2}.
In Section~8 we prove the other assertions of Theorem~\ref{Th2} and Corollary~\ref{Cor1}.

\section{Preliminaries}
\subsection{Regular and singular solutions of the limit equation}
Let us consider the radial solution $\bar{u}(\rho)$ of the elliptic equation on $\R^N$
\begin{equation}\label{S21E0}
\begin{cases}
\Delta \bar{u}+\bar{u}^p=0 & \textrm{in}\ \ \R^N,\\
\bar{u}>0 & \textrm{in}\ \ \R^N.
\end{cases}
\end{equation}
It is well known that
\begin{equation}\label{singent}
\bar{u}^*(\rho):=A\rho^{-\theta}
\end{equation}
is a singular solution of (\ref{S21E0}), where $\theta$ is defined by (\ref{theta}) and $A:=A(p,N)$ is defined by (\ref{Th2E2}).
Next, we consider the regular radial solution $\bar{u}(\rho)$.
The solution $\bar{u}(\rho)$ satisfies (\ref{S1E7+}).
\begin{proposition}\label{S2P1}
Suppose that $p>p_S$. Then the problem (\ref{S1E7+}) has a unique (positive) solution $\bar{u}(\rho,\gamma)$ and $\bar{u}(\rho,\gamma)=\gamma \bar{u}(\gamma^{\frac{p-1}{2}}\rho,1)$. Moreover, if $p_S<p<p_{JL}$, then, for each $\gamma>0$, $\calZ_{(0,\infty)}[\bar{u}^*(\,\cdot\,)-\bar{u}(\,\cdot\, ,\gamma)]=\infty$.
\end{proposition}

\begin{proof}
We define $\bar{v}(\rho):=A^{-1}\rho^{\theta}\bar{u}(\rho)$.
The function $\bar{v}(\rho)$ satisfies
\[
\begin{cases}
\bar{v}''+\frac{1}{\rho}\left( N-1-2\theta\right)\bar{v}'+\frac{A^{p-1}}{\rho^2}(\bar{v}^p-\bar{v})=0, & 0<\rho<\infty ,\\
r^{-\theta}\bar{v}(\rho)\rightarrow\frac{\gamma}{A}\quad (\rho\downarrow 0), \\
\bar{v}>0, & 0\le \rho<\infty.\\
\end{cases}
\]
In order to make the equation autonomous we change variables to $t:=A^{\frac{p-1}{2}}\log\rho$ and $y(t):=\bar{v}(\rho)$.
The function $y(t)$ satisfies
\begin{equation}\label{S21Eq2}
\begin{cases}
y''+\alpha y'-y+y^p=0, & -\infty<t<\infty,\\
e^{-\theta mt}y(t)\rightarrow\frac{\gamma}{A}\quad (t\rightarrow -\infty), & \\
y>0, & -\infty<t<\infty,\\
\end{cases}
\end{equation}
where the prime stands for the derivative,
\[
m:=A^{-\frac{p-1}{2}},\quad\textrm{and}\quad\alpha:=m(N-2-2\theta).
\]
The equation (\ref{S21Eq2}) is transformed into the first order autonomous system
\begin{equation}\label{S21E2+}
\begin{cases}
y'=z &\\
z'=-\alpha z+y-y^p. &
\end{cases}
\end{equation}
This system has two equilibrium points $(0,0)$ and $(1,0)$. By simple calculation we see that $(0,0)$ is a saddle point.
The eigenvalues of the linearized operator at $(0,0)$ are the roots of the quadratic equation $\Lambda^2+\alpha\Lambda -1=0$.
The eigenvalues are $m\theta(>0)$ and $-m(N-2-\theta)(<0)$.
On the other hand, the eigenvalues of the linearized operator at $(1,0)$ are the roots of $\Lambda^2+\alpha\Lambda +p-1=0$.
Therefore, $(1,0)$ is a spiral point if and only if
\[
\alpha^2-4(p-1)<0,\textrm{ i.e., }(N-2-2\theta)^2-8(N-2-\theta)<0.
\]
Then,
\begin{equation}\label{S2Eq22+}
(N-4-2\sqrt{N-1})/2<\theta<(N-4+2\sqrt{N-1})/2.
\end{equation}
We easily see that $1+4/(N-4+2\sqrt{N-1})<p_S(<p)$.
Thus, $\theta<(N-4+2\sqrt{N-1})/2$ always holds.
If $N\le 10$, then $(N-4-2\sqrt{N-1})/2\le 0(<\theta)$, hence (\ref{S2Eq22+}) holds. When $N\ge 11$, (\ref{S2Eq22+}) holds if
\[
p<1+\frac{4}{N-4-2\sqrt{N-1}}(=p_{JL}).
\]
If $p\ge p_{JL}$, then two eigenvalues are negative.
Thus, we have the following:
\[
\begin{cases}
\textrm{If $N\le 10$, then $(1,0)$ is a spiral point for $p>p_S$.} &\\
\textrm{If $N\ge 11$, then $(1,0)$ is a} 
\begin{cases}
\textrm{spiral point for $p_S<p<p_{JL}$.} &\\
\textrm{(resp. degenerate) node for}\\
\textrm{(resp. $p=p_{JL}$) $p>p_{JL}.$} &
\end{cases}
&
\end{cases}
\]
Let
\begin{equation}\label{S2E22++}
E(y,z):=\frac{1}{2}z^2-\frac{1}{2}y^2+\frac{1}{p+1}y^{p+1}.
\end{equation}
By direct calculation we have $\frac{d}{dt}E(y(t),z(t))=-\alpha z^2(t)$.
If $p>p_S$, then $\alpha>0$ and $E$ is a Lyapunov function.
A phase plane analysis with this Lyapunov function reveals that there is a heteroclinic orbit from $(0,0)$ to $(1,0)$ which is in the right half-plane.
Note that the eigenvalue $\theta m$ is compatible with the limit condition in (\ref{S21Eq2}), because $y(t)\sim C_0e^{\theta mt}$ $(t\rightarrow -\infty)$.
This orbit corresponds to a regular solution $\bar{u}(\rho)$ of (\ref{S1E7+}) for some $\gamma$, and $\bar{u}$ is positive.
Since $\bar{u}(\rho)$ is a solution, $\gamma\bar{u}(\gamma^{\frac{p-1}{2}}\rho)$ is also a solution.
Without loss of generality, we assume that $\bar{u}(0)=1$.
Every solution of (\ref{S1E7+}) can be written as $\gamma\bar{u}(\gamma^{\frac{p-1}{2}}\rho)$. 
Since the equilibrium point $(1,0)$ corresponds to the singular solution $\bar{u}^*(\rho)$ of (\ref{S1E7+}) and $(1,0)$ is a spiral point for $p_S<p<p_{JL}$, the intersection number of $\bar{u}(\rho,\gamma)$ and $\bar{u}^*(\rho)$ is infinite for each $\gamma$, i.e., $\calZ_{(0,\infty)}[\bar{u}^*(\,\cdot\,)-\bar{u}(\,\cdot\, ,\gamma)]=\infty$.
\end{proof}

\subsection{Nonexistence of the entire solution}
\begin{proposition}\label{S22P1}
If $u(s)\in C^2(\e,R)\cap C^0[\e,R]$ satisfies
\[
u''+\frac{N-1}{s}u'+f(u)=0\quad (\e<s<R)
\]
and $u(s)>0$ $(\e<s<R)$, then
\begin{multline}\label{S22P1E1}
0=\left(\frac{N}{2}-1-\frac{N}{p+1}\right)\int_{\e}^Rs^{N-1}(u')^2ds+\left(\frac{N}{2}-\frac{N}{p+1}\right)\int_{\e}^Rs^{N-1}u^2ds\\
+\frac{N}{p+1}\left( R^{N-1}u(R)u'(R)-\e^{N-1}u(\e)u'(\e)\right)
+\frac{1}{2}\left(R^Nu'(R)^2-\e^Nu'(\e)^2\right)\\
-\frac{1}{2}\left(R^Nu(R)^2-\e^Nu(\e)^2\right)
+\frac{1}{p+1}\left(R^Nu(R)^{p+1}-\e^Nu(\e)^{p+1}\right).
\end{multline}
\end{proposition}

This is the integral form of the Pohozaev identity.
We omit the proof.
See \cite[Proposition~4.3]{NY88} or \cite[Lemma~3.7]{DN85} for detail.

\begin{proposition}\label{S22P2}
Suppose that $p\ge p_S$. The problem
\[
\begin{cases}
u''+\frac{N-1}{s}u'+f(u)=0, & 0<s<\infty, \\
u'(0)=0,\ u(0)<\infty,\ u(\infty)=0,\\
u>0, & 0\le s<\infty
\end{cases}
\]
does not admit a solution.
\end{proposition}

This proposition was obtained in \cite[Theorem~1.6]{N83}.
We omit the proof.

\section{Local and global parametrization results}
In this section we study the parametrization of each branch $\calC_n$ of the solution of (\ref{N}).
Let $U(r)$ be a regular solution of (\ref{N}).
We sometimes use the stretched variable $s:=\sqrt{\lambda}r(=r/\e)$.
We define $u(s):=U(r)$.
Then $u(s)$ satisfies (\ref{S1E7}).
We denote the solution of the initial value problem (\ref{S3E2}) by $u(s,\gamma)$.
\subsection{Local parametrization result}
To begin with, we study the local parametrization of the branch.
\begin{proposition}[{{\cite[Proposition~3.1]{MY13}}}]\label{LPR}
Let $(\lambda_0,U_0)$ be a nonconstant regular solution of (\ref{N}), and let $\gamma_0:=U_0(0)$.
Then, all solutions near $(\lambda_0,U_0)$ can be parametrized as $\{(\lambda(\gamma),U(r,\gamma))\}_{|\gamma-\gamma_0|<\e}$ ($U(0,\gamma)=\gamma$, $\lambda(\gamma_0)=\lambda_0$, $U(r,\gamma_0)=U_0(r)$).
\end{proposition}
This proposition was obtained in \cite{MY13}.
We prove this proposition for readers' convenience.\\

\noindent
\begin{proof}[Proof of Proposition~\ref{LPR}]
Let $s:=\sqrt{\lambda}r$ and $u(s):=U(r)$.
Then, $u$ is not constant and $u$ satisfies (\ref{S1E7}).
We consider the initial value problem (\ref{S3E2}) and denote the solution by $u(s,\gamma)$.
Then, $u$ is a $C^2$-function of $(s,\gamma)$ and $U(r,\lambda,\gamma)(=u(\sqrt{\lambda}r,\gamma))$ satisfies (\ref{N}).
Since $U_r(1,\lambda,\gamma)=0$, $\sqrt{\lambda}u_s(\sqrt{\lambda},\gamma)=0$. We will show that $\frac{\partial}{\partial\lambda}u_s(\sqrt{\lambda},\gamma)\neq 0$.
Since $\frac{\partial}{\partial\lambda}u_s(\sqrt{\lambda},\gamma)=u_{ss}(\sqrt{\lambda},\gamma)/(2\sqrt{\lambda})$, it is enough to show that $u_{ss}(\sqrt{\lambda},\gamma)\neq 0$.
Suppose the contrary, i.e., $u_{ss}(\sqrt{\lambda},\gamma)=0$.
Differentiating (\ref{S3E2}) with respect to $s$, we have
\[
u'''+\frac{N-1}{s}u''+\left( f'(u)-\frac{N-1}{s^2}\right)u'=0.
\]
Since $u_s(\sqrt{\lambda},\gamma)=u_{ss}(\sqrt{\lambda},\gamma)=0$, we see by the uniqueness of the solution of the ODE that $u_s(s,\gamma)\equiv 0$.
Thus, $u$ is a constant solution of (\ref{S3E2}) which contradicts that $U$ is not constant.
Hence, $u_{ss}(\sqrt{\lambda},\gamma)\neq 0$.
We can apply the implicit function theorem to $u_s(\sqrt{\lambda},\gamma)=0$.
We see that there is a $C^1$-function $\lambda=\lambda(\gamma)$, which is defined in a neighborhood of $\gamma_0$, such that $u_s(\sqrt{\lambda(\gamma)},\gamma)=0$ and that all the solutions near $(\lambda_0,U_0)$ can be written as $(\lambda(\gamma),u(\sqrt{\lambda(\gamma)}r,\gamma))$ $(|\gamma-\gamma_0|<\e)$.
The positivity $u(\sqrt{\lambda(\gamma)}r,\gamma)>0$ $(0\le r\le 1)$ follows from the maximum principle.
\end{proof}

\begin{remark}
It follows from Proposition~\ref{LPR} that another radial branch does not emanate from $\calC_n$. However, the branch may have a turning point.
\end{remark}

\subsection{Local bifurcation from the trivial branch}
We work on the space of radial functions.
Let $\phi_n$ be the eigenfunction of the linearization of (\ref{N}) at $(\bar{\lambda}_n,1)$.
Since
\[
\int_0^1\phi_n\left.\frac{\partial^2}{\partial U\partial\lambda}(\lambda f(U))\right|_{(\lambda,U)=(\bar{\lambda}_n,1)}\phi_n r^{N-1}dr=(p-1)\int_0^1\phi_n^2 r^{N-1}dr\neq 0,
\]
the transversality condition of Crandall-Rabinowitz type holds, hence $(\bar{\lambda}_n,1)$ is a bifurcation point.

\subsection{Extension to $1<\gamma<\infty$}
For each $n\ge 1$, the branch $\calC_n$ emanates from $(\bar{\lambda}_n,1)$ and it is locally parametrized by $\gamma=U_n(0)$, hence $\calC_n=\{(\lambda_n(\gamma),U_n(r,\gamma))\}$ $(\gamma=U_n(0,\gamma))$.
In this subsection we extend the branch in the direction $\gamma\rightarrow\infty$.
In particular, we show that, for each $\gamma^*>1$, $\lambda_n(\gamma)$ does not diverge as $\gamma\uparrow\gamma^*$.
\begin{lemma}\label{BDDB}
For each $\gamma^*>1$, there are $C>0$ and $\e>0$ such that $|\lambda_n(\gamma)|<C$ for $\gamma\in(\gamma^*-\e,\gamma^*)$.
\end{lemma}

\begin{proof}
First, we prove the case $n=1$.
We use a contradiction argument.
Suppose the contrary, i.e., 
\begin{multline}\label{BDDBE0}
\textrm{there is a sequence $\{\gamma_j\}_{j=1}^{\infty}$ 
$(\gamma_j<\gamma^*,\ \gamma_j\uparrow\gamma^*)$}\\
\textrm{such that $\lambda_1(\gamma_j)\rightarrow\infty$ $(j\rightarrow\infty)$.}
\end{multline}
Let $u(s,\gamma_j)$ denote the solution of (\ref{S3E2}), and let $u_j(s):=u(s,\gamma_j)$.
Then, $u_j(0)=\gamma_j$.
We show that there exists a solution $u_*(s)\in C^2(0,\infty)\cap C^0[0,\infty)$ of the problem
\begin{equation}\label{BDDBE1}
\begin{cases}
u''+\frac{N-1}{s}u'+f(u)=0, & 0<s<\infty,\\
u'(0)=0,\ u(0)=\gamma^*,\ u(\infty)=0, & \\
u>0,\ u'<0, & 0<s<\infty.
\end{cases}
\end{equation}
Since the problem (\ref{S3E2}) is well-posed, there exists $\tilde{u}_*(s)\in C^2(0,\infty)\cap C^0[0,\infty)$ such that $\tilde{u}_*$ satisfies
\[
\begin{cases}
(\tilde{u}_*)''+\frac{N-1}{s}(\tilde{u}_*)'+f(\tilde{u}_*)=0, & 0<s<\infty,\\
(\tilde{u}_*)'(0)=0,\ \tilde{u}_*(0)=\gamma^*, &
\end{cases}
\]
and $u_j(s)\rightarrow\tilde{u}_*(s)$ in $C^2_{loc}(0,\infty)\cap C^0_{loc}[0,\infty)$ as $j\rightarrow\infty$.
Therefore, in order to show that $\tilde{u}_*(s)\equiv u_*(s)$ we show that $\tilde{u}_*$ satisfies
\begin{equation}\label{BDDBE2}
\begin{cases}
\tilde{u}_*(\infty)=0, & \\
\tilde{u}_*>0,\ (\tilde{u}_*)'<0, & 0<s<\infty.
\end{cases}
\end{equation}
If there is $\tilde{s}_0>0$ such that $\tilde{u}_*(\tilde{s}_0)<0$, then, for large $j\ge 1$, $u_j(\tilde{s}_0)<0$, since $u_j\rightarrow\tilde{u}_*$ in $C^0_{loc}[0,\infty)$.
We obtain a contradiction, because $\tilde{s}_0<\sqrt{\lambda_1(\gamma_j)}$ for large $j\ge 1$ and $u_j(s)>0$ $(0<s<\sqrt{\lambda_1(\gamma_j)})$ for every $j\ge 1$.
Thus $\tilde{u}_*(s)\ge 0$ $(0\le s<\infty)$.
By the maximum principle we see that $\tilde{u}_*(s)>0$ $(0\le s<\infty)$.
If there is $\tilde{s}_1>0$ such that $(\tilde{u}_*)'(\tilde{s}_1)>0$, then, for large $j\ge 1$, $u'_j(\tilde{s}_1)>0$.
We obtain a contradiction, because $\tilde{s}_1<\sqrt{\lambda_1(\gamma_j)}$ for large $j\ge 1$ and $u'_j(s)<0$ $(0<s<\sqrt{\lambda_1(\gamma_j)})$. Thus, $(\tilde{u}_*)'(s)\le 0$ $(0<s<\infty)$.
Let $v:=(\tilde{u}_*)'$. If there is $\tilde{s}_2>0$ such that $v(\tilde{s}_2)=0$, then $v'(\tilde{s}_2)=0$, since $v(s)\le 0$ ($0<s<\infty$). We obtain a contradiction, because $v$ satisfies
\[
v''+\frac{N-1}{s}v'+\left( f'(\tilde{u}_*)-\frac{N-1}{s^2}\right)v=0\quad (0<s<\infty)
\]
and Hopf's lemma says that $v'(\tilde{s}_2)\neq 0$. Therefore, $v(s)<0$ $(0<s<\infty)$ and $(\tilde{u}_*)'(s)<0$ $(0<s<\infty)$.
Since $\tilde{u}_*>0$ and $(\tilde{u}_*)'<0$, $\tilde{u}_*(\infty)$ exists and $(\tilde{u}_*)'(\infty)=0$.
Since $\tilde{u}_*$ satisfies (\ref{S3E2}), $(\tilde{u}_*)''(\infty)$ exists and it follows from the boundedness of $\tilde{u}^*$ that $(\tilde{u}_*)''(\infty)=0$. Because of (\ref{S3E2}), $f(\tilde{u}_*(\infty))=0$, hence $\tilde{u}_*(\infty)=0$ or $1$. Suppose that $\tilde{u}_*(\infty)=1$. Let $w:=u-1$. Then $w$ satisfies
\begin{equation}\label{BDDBE2+}
w''+\frac{N-1}{s}w'+\frac{(w+1)^p-(w+1)}{w}w=0.
\end{equation}
If $w$ is close to $0$, then $\{(w+1)^p-(w+1)\}/w$ is close to $p-1$.
Since $p-1>0$, by Sturm's oscillation theorem we see that $w$ oscillates around $0$ which contradicts that $w'<0$.
Thus, $\tilde{u}_*(\infty)=0$.
We have shown that (\ref{BDDBE2}) holds and $\tilde{u}_*\equiv u_*$.

By Proposition~\ref{S22P2} we see that if $p\ge p_S$, then (\ref{BDDBE1}) has no solution.
We obtain a contradiction.
Then, (\ref{BDDBE0}) does not hold, and $|\lambda_1(\gamma)|$ is bounded for $\gamma\in (\gamma^*-\e,\gamma^*)$.
Therefore, there exists a subsequence, which is still denoted by $\{\gamma_j\}$, such that $\lambda_1(\gamma_j)$ converges as $j\rightarrow\infty$. We define $\lambda_1(\gamma^*)=\lim_{j\rightarrow\infty}\lambda_1(\gamma_j)$. Because of Proposition~\ref{LPR}, $\lambda_1(\gamma^*)$ does not depend on the subsequence and it is uniquely determined.
Moreover, $u_s(\sqrt{\lambda_1(\gamma^*)},\gamma^*)=0$.

Second, we prove the case $n\ge 2$.
Let $E$ be as given by (\ref{S2E22++}).
Then
\begin{equation}\label{BDDBE3}
\frac{d}{ds}E(u(s),v(s))=-\frac{N-1}{s}v^2\le 0.
\end{equation}
Therefore, $E(u(s),v(s))\le E(u(\sqrt{\lambda_1(\gamma^*)}),0)<0$ ($s>\sqrt{\lambda_1(\gamma^*)}$), which implies that 
\begin{equation}\label{BDDBE4}
u(s)\ge u(\sqrt{\lambda_1(\gamma^*)})(>0).
\end{equation}
Let $w(s):=u(s)-1$. Then $w$ satisfies (\ref{BDDBE2+}).
Because of (\ref{BDDBE4}), there is $\delta>0$ such that 
\[
\frac{(w+1)^p-(w+1)}{w}>\frac{u(\sqrt{\lambda_1(\gamma^*)}))^p-u(\sqrt{\lambda_1(\gamma^*)})}{u(\sqrt{\lambda_1(\gamma^*)})-1}>\delta >0\  (\gamma>0)
\]
provided that $s>\sqrt{\lambda_1(\gamma^*)}$.
Sturm's oscillation theorem says that $w$ oscillates around $0$ infinitely many times.
Therefore, $u(s)$ has the $n$-th positive critical point $s_n$.
Since $\lambda_n(\gamma^*)=s_n^2$ and $\lambda_n(\gamma)$ is continuous at $\gamma^*$ (Proposition~\ref{LPR}), we see that $|\lambda_n(\gamma)|<C$ for $\gamma\in (\gamma^*-\e,\gamma^*)$.
\end{proof}

Using Lemma~\ref{BDDB}, we prove the global parametrization result.
\begin{lemma}\label{GPR}
For each $n\ge 1$, $\{(\lambda_n(\gamma),U_n(\gamma))\}$ can be defined in $1<\gamma<\infty$.
\end{lemma}
\begin{proof}
Let $\calC_n:=\{(\lambda_n(\gamma),U_n(\gamma))\}$ be the branch which emanates from $(\bar{\lambda}_n,1)$. We extend $\calC_n$. We define
\[
\gamma^*:=\sup\{\bar{\gamma}>1;\ \lambda_n(\gamma)<\infty\ \textrm{for every}\ \gamma\in(1,\bar{\gamma}).\}.
\]
Suppose the contrary, i.e., $\gamma^*<\infty$. Let $\{\gamma_j\}_{j=1}^{\infty}$ $(\gamma_j<\gamma^*,\ \gamma_j\uparrow\gamma^*)$ be a sequence.
Let $U_{n,j}(r)$ be a solution of (\ref{N}) such that $U_{n,j}(0)=\gamma_j$ and $U_{n,j}'(0)=0$.
Because of Lemma~\ref{BDDB}, there are $\lambda_*\ge 0$ and a subsequence of $\{\gamma_j\}$, which is still denoted by $\{\gamma_j\}$, such that $\lambda_n(\gamma_j)\rightarrow\lambda_*$ $(j\rightarrow\infty)$.
Note that it is clear that $\lambda_*>0$.
Because of the continuous dependence of the solution $U(r,\gamma,\lambda)$ of
\begin{equation}\label{GRPE1}
\begin{cases}
U''+\frac{N-1}{r}U'+\lambda f(U)=0, & 0<r<1,\\
U'(0)=0,\ U(0)=\gamma &
\end{cases}
\end{equation}
on $(\gamma,\lambda)$ in $C^1_{loc}[0,\infty)$, there exists a solution $U_*$ of (\ref{GRPE1}) with $(\gamma,\lambda)=(\gamma^*,\lambda_*)$ such that $(U_*)'(1)=0$.
Since each zero of $U_n(\,\cdot\,,\gamma)-1$ is simple, $\calZ_{(0,1)}[U_*(\,\cdot\,)-1]=\calZ_{(0,1)}[U_n(\,\cdot\,,\gamma_j)-1]=n$ for large $j\ge 0$.
Therefore, $(\lambda_*,U_*)$ is a solution of (\ref{N}) and $(\lambda_*,U_*)\in \calC_n$.
By Proposition~\ref{LPR} we can extend the branch $\calC_n$ to $1<\gamma<\gamma^*+\e$ for small $\e>0$.
This contradicts the definition of $\gamma^*$.
Thus, $\gamma^*=\infty$.
\end{proof}
The branch $\calC^+_n$ is unbounded in $\gamma$.

\subsection{Extension to $0<\gamma<1$}
We extend $\calC^-_n$.
\begin{lemma}\label{GPR2}
For each $n\ge 1$, $\{(\lambda_n(\gamma),U_n(\gamma))\}$ can be defined in $0<\gamma<1$.
\end{lemma}

\begin{proof}
Let
\[
\gamma_*:=\inf\{\underline{\gamma}>0;\ \lambda_n(\gamma)<\infty\ \textrm{for every}\ \underline{\gamma}<\gamma<1\}.
\]
Suppose the contrary, i.e., $\gamma_*>0$. 
Let $u(s)$ be the solution of the (\ref{S3E2}) with $\gamma=\underline{\gamma}$, and let $v(s):=u'(s)$.
Let $E$ be as given by (\ref{S2E22++}).
Then (\ref{BDDBE3}) holds.
Therefore, $E(u(s),v(s))\le E(\underline{\gamma},0)<0$, which implies that 
\begin{equation}\label{GRP2E1}
u(s)\ge u(0)(=\underline{\gamma}>0).
\end{equation}
Let $w(s):=u(s)-1$. Then $w$ satisfies (\ref{BDDBE2+}).
Because of (\ref{GRP2E1}), there is $\delta>0$ such that $\{(w+1)^p-(w+1)\}/w>(u(0)^p-u(0))/(u(0)-1)>\delta >0$ ($\gamma>0$).
Sturm's oscillation theorem says that $w$ oscillates around $0$ infinitely many times.
Therefore, $u(s)$ has the $n$-th positive critical point $s_n$.
Let $\lambda_n:=s_n^2$. Then $(\lambda(\underline{\gamma}),u)$ is a solution of (\ref{S3E2}) with $\gamma=\underline{\gamma}$.
By Proposition~\ref{LPR} we can extend $\lambda_n(\underline{\gamma})$, which contradicts the definition of $\underline{\gamma}$.
Thus, $\underline{\gamma}=0$, and $\lambda_n(\gamma)$ can be defined in $0<\gamma<1$.
\end{proof}
The branch $\calC^-_n$ is defined in $0<\gamma<1$.

\section{Entire singular solution}
In this section we prove Theorem~\ref{Th1}. Let $(\lambda,U(r))$ be a solution of (\ref{N}). Then $u(s):=U(r)$ $(s:=\sqrt{\lambda}r)$ satisfies (\ref{S1E7}). We use the same change of variables $y(t):=A(p,N)^{-1}s^{\theta}u(s)$ and $t:=m^{-1}\log s$ as in Section~2. Here
\[
m:=\left\{\theta(N-2-\theta)\right\}^{-\frac{1}{2}},
\]
and $\theta$ is defined by (\ref{theta}).
The function $y(t)$ satisfies
\begin{equation}\label{S4E1}
y''+\alpha y'-y+y^p-m^2e^{2mt}y=0,
\end{equation}
where $\alpha:=m(N-2-2\theta)$. We mainly consider (\ref{S4E1}) in this section.

\subsection{Existence of the singular solution near $r=0$}
We construct the singular solution of (\ref{S1E7}) near $s=0$.
\begin{lemma}\label{S41L0}
The problem
\begin{equation}\label{S41L0E0}
\begin{cases}
y''+\alpha y'-y+y^p-m^2e^{2mt}y=0, \\
y(t)\rightarrow 1\quad (t\rightarrow -\infty)
\end{cases}
\end{equation}
has a unique solution. 
\end{lemma}
The proof is essentially the same as one of \cite[Theorem~A]{MP91}.
We use the following lemma to prove Lemma~\ref{S41L0}.
\begin{lemma}\label{S41L1}
Assume that (\ref{S41L0E0}) has a solution $y^*(t)$. Then, $y^*(t)$ satisfies
\[
y^*(t)=1+O(e^{2mt})\quad (t\rightarrow -\infty).
\]
\end{lemma}

\begin{proof}
Let $\zeta:=-t$ and $\eta(\zeta):=y(t)-1$.
The function $\eta(\zeta)$ satisfies
\begin{equation}\label{S41L1E1}
\begin{cases}
\eta''-\alpha\eta'+(p-1)\eta=g(\zeta),& \zeta_0<\zeta<\infty,\\
\eta(\zeta)\rightarrow 0\quad (\zeta\rightarrow\infty),
\end{cases}
\end{equation}
where $\zeta_0$ is a large number,
\begin{equation}\label{S41L1E2}
g(\zeta):=m^2e^{-2m\zeta}(1+\eta(\zeta))-\varphi(\eta(\zeta)),
\end{equation}
\begin{equation}\label{S41L1E3}
\varphi(\eta):=(1+\eta)^p-1-p\eta.
\end{equation}
There are three possibilities:
\begin{equation}\label{S41L1E3+}
\textrm{(1)}\ p-1>\left(\frac{\alpha}{2}\right)^2,\quad
\textrm{(2)}\ p-1<\left(\frac{\alpha}{2}\right)^2,\quad
\textrm{(3)}\ p-1=\left(\frac{\alpha}{2}\right)^2.
\end{equation}

\noindent
{\it Case (1):} Let $\beta:=\sqrt{(p-1)-\left(\frac{\alpha}{2}\right)^2}$.
Because the linearly independent solutions of the homogeneous equation associated with the equation in (\ref{S41L1E1}) becomes unbounded as $\zeta\rightarrow\infty$, we have
\begin{equation}\label{S41L1E4}
\eta(\zeta)=\frac{e^{\frac{\alpha}{2}\zeta}}{\beta}\int_{\zeta}^{\infty}e^{-\frac{\alpha}{2}\sigma}\sin(\beta(\sigma-\zeta))g(\sigma)d\sigma.
\end{equation}
If $|\eta|$ is small, then there are $\e>0$ and $\zeta_{\e}>0$ such that
\begin{equation}\label{S41L1E5}
|\varphi(\eta)|\le|(1+\eta)^p-1-p\eta|\le\e|\eta|\quad (\zeta>\zeta_{\e}).
\end{equation}
By (\ref{S41L1E2}) and (\ref{S41L1E5}) we have
\begin{equation}\label{S41L1E6}
|g(\zeta)|\le C_0e^{-2m\zeta}+\e|\eta(\zeta)|\quad (\zeta>\zeta_{\e}).
\end{equation}
Using (\ref{S41L1E4}), we have
\begin{align}
\int_{\zeta}^{\infty}|\eta(\sigma)|d\sigma &\le\int_{\zeta}^{\infty}\frac{e^{\frac{\alpha}{2}\sigma}}{\beta}\int_{\sigma}^{\infty}e^{-\frac{\alpha}{2}\tau}|g(\tau)|d\tau d\sigma\nonumber\\
&=\left[\frac{2e^{\frac{\alpha}{2}\sigma}}{\alpha\beta}\int_{\sigma}^{\infty}e^{-\frac{\alpha}{2}\tau}|g(\tau)|d\tau\right]_{\zeta}^{\infty}+\int_{\zeta}^{\infty}\frac{2e^{\frac{\alpha}{2}\sigma}}{\alpha\beta}e^{-\frac{\alpha}{2}\sigma}|g(\sigma)|d\sigma\nonumber\\
&\le -\frac{2e^{\frac{\alpha}{2}\zeta}}{\alpha\beta}\int_{\zeta}^{\infty}e^{-\frac{\alpha}{2}\sigma}|g(\sigma)|d\sigma+\frac{2}{\alpha\beta}\int_{\zeta}^{\infty}|g(\sigma)|d\sigma\nonumber\\
&\le\frac{2}{\alpha\beta}\int_{\zeta}^{\infty}|g(\sigma)|d\sigma,\label{S41L1E7}
\end{align}
where we use
\[
\lim_{\sigma\rightarrow\infty}\frac{2\int_{\sigma}^{\infty}e^{-\frac{\alpha}{2}\tau}|g(\tau)|d\tau}{\alpha\beta e^{-\frac{\alpha}{2}\sigma}}
=\lim_{\sigma\rightarrow\infty}\frac{-2e^{-\frac{\alpha}{2}\sigma}|g(\sigma)|}{\alpha\beta\left(-\frac{\alpha}{2}\right)e^{-\frac{\alpha}{2}\sigma}}=0.
\]
By (\ref{S41L1E6}) and (\ref{S41L1E7}) we have
\[
\int_{\zeta}^{\infty}|\eta(\sigma)|d\sigma\le\frac{2}{\alpha\beta}\int_{\zeta}^{\infty}\left( C_0e^{-2m\sigma}+\e|\eta(\sigma)|\right)d\sigma\quad (\zeta>\zeta_{\e}).
\]
Hence,
\begin{align}
\left(1-\frac{2\e}{\alpha\beta}\right)\int_{\zeta}^{\infty}|\eta(\sigma)|d\sigma
&\le\frac{2C_0}{\alpha\beta}\int_{\zeta}^{\infty}e^{-2m\sigma}d\sigma\quad \nonumber\\
&=\frac{C_0e^{-2m\zeta}}{\alpha\beta m}\quad(\zeta>\zeta_{\e}).\label{S41L1E8}
\end{align}
On the other hand, by (\ref{S41L1E4}) and (\ref{S41L1E6}) we have
\begin{align}
|\eta(\zeta)|&\le
\frac{e^{\frac{\alpha}{2}\zeta}}{\beta}\int_{\zeta}^{\infty}e^{-\frac{\alpha}{2}\sigma}\left( C_0e^{-2m\sigma}+\e|\eta(\sigma)|\right)d\sigma\nonumber\\
&=\frac{C_0}{\beta}e^{\frac{\alpha}{2}\zeta}
\int_{\zeta}^{\infty}e^{-\left(\frac{\alpha}{2}+2m\right)\sigma}d\sigma
+\frac{\e}{\beta}\int_{\zeta}^{\infty}|\eta(\sigma)|d\sigma\nonumber\\
&=\frac{C_0e^{-2m\zeta}}{\beta\left(\frac{\alpha}{2}+2m\right)}+\frac{\e}{\beta}\int_{\zeta}^{\infty}|\eta(\sigma)|d\sigma.\label{S41L1E9}
\end{align}
By (\ref{S41L1E8}) and (\ref{S41L1E9}) we have
\begin{align}
|\eta(\zeta)| &\le\frac{C_0e^{-2m\zeta}}{\beta\left(\frac{\alpha}{2}+2m\right)}
+\frac{C_0\e}{\left( 1-\frac{2\e}{\alpha\beta}\right)\alpha\beta^2m}e^{-2m\zeta}\nonumber\\
&=:C_1e^{-2m\zeta}.\label{S41L1E10}
\end{align}

\noindent
{\it Case (2):} Let $\beta:=\sqrt{\left(\frac{\alpha}{2}\right)^2-(p-1)}$.
Then, $\frac{\alpha}{2}-\beta>0$. We have
\begin{equation}\label{S41L1E11}
\eta(\zeta)=\frac{e^{\frac{\alpha}{2}\zeta}}{\beta}\int_{\zeta}^{\infty}e^{-\frac{\alpha}{2}\sigma}\sinh (\beta(\sigma-\zeta))g(\sigma)d\sigma.
\end{equation}
Using (\ref{S41L1E11}) and $|\sinh(\beta(\tau-\sigma))|\le e^{\beta(\tau-\sigma)}/2$ $(\tau\ge\sigma)$, we have

\begin{align}
\int_{\zeta}^{\infty}|\eta(\sigma)|d\sigma
&\le
\int_{\zeta}^{\infty}\frac{e^{\frac{\alpha}{2}\sigma}}{\beta}\int_{\sigma}^{\infty}e^{-\frac{\alpha}{2}\tau}|\sinh(\beta(\tau-\sigma))||g(\tau)|d\tau d\sigma\nonumber\\
&\le\int_{\zeta}^{\infty}\frac{e^{\frac{\alpha}{2}\sigma}}{\beta}\int_{\sigma}^{\infty}e^{-\frac{\alpha}{2}\tau}\frac{e^{\beta(\tau-\sigma)}}{2}|g(\tau)|d\tau d\sigma\nonumber\\
&=\int_{\zeta}^{\infty}\frac{e^{\left(\frac{\alpha}{2}-\beta\right)\sigma}}{2\beta}\int_{\sigma}^{\infty}e^{-\left(\frac{\alpha}{2}-\beta\right)\tau}|g(\tau)|d\tau d\sigma\nonumber\\
&=\left[\frac{e^{\left(\frac{\alpha}{2}-\beta\right)\sigma}}{\beta(\alpha-2\beta)}\int_{\sigma}^{\infty}e^{-\left(\frac{\alpha}{2}-\beta\right)\tau}|g(\tau)|d\tau\right]_{\zeta}^{\infty}\nonumber\\
&\quad+\frac{1}{\beta(\alpha-2\beta)}\int_{\zeta}^{\infty}|g(\sigma)|d\sigma\nonumber\\
&=-\frac{e^{\left(\frac{\alpha}{2}-\beta\right)\zeta}}{\beta(\alpha-2\beta)}\int_{\zeta}^{\infty}e^{-\left(\frac{\alpha}{2}-\beta\right)\tau}|g(\tau)|d\tau\nonumber\\
&\quad+\frac{1}{\beta(\alpha-2\beta)}\int_{\zeta}^{\infty}|g(\sigma)|d\sigma\nonumber\\
&\le\frac{1}{\beta(\alpha-2\beta)}\int_{\zeta}^{\infty}|g(\sigma)|d\sigma,\label{S41L1E12}
\end{align}
where we use $\frac{\alpha}{2}-\beta>0$ and
\[
\lim_{\sigma\rightarrow\infty}
\frac{\int_{\sigma}^{\infty}e^{-\left(\frac{\alpha}{2}-\beta\right)\tau}|g(\tau)|d\tau}{\beta(\alpha-2\beta)e^{-\left(\frac{\alpha}{2}-\beta\right)\sigma}}
=\lim_{\sigma\rightarrow\infty}\frac{-e^{-\left(\frac{\alpha}{2}-\beta\right)\sigma}|g(\sigma)|}{\beta(\alpha-2\beta)\left(-\frac{\alpha}{2}+\beta\right)e^{-\left(\frac{\alpha}{2}-\beta\right)\sigma}}=0.
\]
In the case (2) (\ref{S41L1E6}) also holds. By (\ref{S41L1E6}) and (\ref{S41L1E12}) we have
\[
\int_{\zeta}^{\infty}|\eta(\sigma)|d\sigma\le\frac{1}{\beta(\alpha-2\beta)}\int_{\zeta}^{\infty}\left( C_0e^{-2m\sigma}+\e|\eta(\sigma)|\right)d\sigma\quad (\zeta>\zeta_{\e}).
\]
Hence,
\begin{align}
\left( 1-\frac{\e}{\beta(\alpha-2\beta)}\right)\int_{\zeta}^{\infty}|\eta(\sigma)|d\sigma
&\le\frac{C_0}{\beta(\alpha-2\beta)}\int_{\zeta}^{\infty}e^{-2m\sigma}d\sigma\nonumber\\
&=\frac{C_0}{2m\beta(\alpha-2\beta)}e^{-2m\zeta}\quad (\zeta>\zeta_{\e}).\label{S41L1E13}
\end{align}
On the other hand, by (\ref{S41L1E11}) and (\ref{S41L1E6}) we have
\begin{align}
|\eta(\zeta)|&\le\frac{e^{\frac{\alpha}{2}\zeta}}{\beta}\int_{\zeta}^{\infty}e^{-\frac{\alpha}{2}\sigma}\frac{e^{\beta(\sigma-\zeta)}}{2}\left( C_0e^{-2m\sigma}+\e|\eta(\sigma)|\right)d\sigma\nonumber\\
&=\frac{C_0}{2\beta}e^{\left(\frac{\alpha}{2}-\beta\right)\zeta}\int_{\zeta}^{\infty}e^{-\left(\frac{\alpha}{2}-\beta+2m\right)\sigma}d\sigma\nonumber\\
&\quad+\frac{\e}{2\beta}e^{\left(\frac{\alpha}{2}-\beta\right)\zeta}\int_{\zeta}^{\infty}e^{-\left(\frac{\alpha}{2}-\beta\right)\sigma}|\eta(\sigma)|d\sigma\nonumber\\
&\le\frac{C_0e^{-2m\zeta}}{2\beta\left(\frac{\alpha}{2}-\beta+2m\right)}+\frac{\e}{2\beta}\int_{\zeta}^{\infty}|\eta(\sigma)|d\sigma.\label{S41L1E14}
\end{align}
By (\ref{S41L1E13}) and (\ref{S41L1E14}) we have
\begin{align}
|\eta(\zeta)|
&\le\frac{C_0e^{-2m\zeta}}{\beta(\alpha-2\beta+4m)}
+\frac{C_0\e}{4m\beta\left(\beta(\alpha-2\beta)-\e\right)}e^{-2m\zeta}\nonumber\\
&=:C_2e^{-2m\zeta}\quad(\zeta>\zeta_{\e}).\label{S41L1E15}
\end{align}

\noindent
{\it Case (3):} In this case we have
\begin{equation}\label{S41L1E16}
\eta(\zeta)=e^{\frac{\alpha}{2}\zeta}\int_{\zeta}^{\infty}e^{-\frac{\alpha}{2}\sigma}(\sigma-\zeta)g(\sigma)d\sigma.
\end{equation}
For each $\beta>0$, we have
\[
|\sigma-\zeta|\le\frac{1}{\beta}|\sinh(\beta(\sigma-\zeta))|.
\]
Hence,
\[
|\eta(\zeta)|\le\frac{e^{\frac{\alpha}{2}\zeta}}{\beta}\int_{\zeta}^{\infty}e^{-\frac{\alpha}{2}\sigma}|\sinh(\beta(\sigma-\zeta))||g(\sigma)|d\sigma.
\]
Therefore, by the same argument as in Case (2) we have (\ref{S41L1E15}). However, $\beta$ is not the same value as in Case (2).

We have verified all the cases.
\end{proof}

\begin{proof}[Proof of Lemma~\ref{S41L0}]
There are three cases (\ref{S41L1E3+}) as in the proof of Lemma~\ref{S41L1}.
We consider only the case (1).
The other cases follow similarly.

We transform (\ref{S41L0E0}) to the integral equation
\begin{equation}\label{S41L0E1}
\eta(\zeta)=\calF(\eta)(\zeta).
\end{equation}
The integral operator $\calF$ is different in each case. In this proof we write $g(\eta,\sigma)$ in order to stress the dependence of $g$ on $\eta$. In the case (1) $\calF$ becomes
\[
\calF(\eta)(\zeta)=\frac{e^{\frac{\alpha}{2}\zeta}}{\beta}\int_{\zeta}^{\infty}e^{-\frac{\alpha}{2}\sigma}\sin(\beta(\sigma-\zeta))g(\eta,\sigma)d\sigma.
\]
We show by contraction mapping theorem that (\ref{S41L0E1}) has a unique solution. Let $\zeta_0>0$ be large. Hereafter, by $\|\,\cdot\,\|$ we denote $\|\,\cdot\,\|_{C^0[\zeta_0,\infty)}$.
We set $X:=\{\eta(\zeta)\in C^0[\zeta_0,\infty);\ \|\eta(\zeta)\|<\infty\}$ and $\calB:=\{\eta(\zeta)\in X;\ \|\eta\|<\bar{\eta}\}$. Let $\eta_1$, $\eta_2\in\calB$. If $\bar{\eta}>0$ is small, then there is a small $\e>0$ such that
\begin{equation}\label{S41L0E2}
|\varphi(\eta_1)-\varphi(\eta_2)|\le\e|\eta_1-\eta_2|.
\end{equation}
Using (\ref{S41L0E2}), we have
\begin{align}
|g(\eta_1,\sigma)-g(\eta_2,\sigma)|
&\le m^2e^{-2m\zeta}|\eta_1-\eta_2|+\e|\eta_1-\eta_2|\nonumber\\
&\le 2\e|\eta_1-\eta_2|,\label{S41L0E3}
\end{align}
provided that $\zeta>0$ is large. By (\ref{S41L0E3}) we have
\begin{align}
\|\calF(\eta_1)-\calF(\eta_2)\|&\le
\left\|\frac{e^{\frac{\alpha}{2}\zeta}}{\beta}\int_{\zeta}^{\infty}e^{-\frac{\alpha}{2}\sigma}|g(\eta_1,\sigma)-g(\eta_2,\sigma)|d\sigma\right\|\nonumber\\
&\le 2\e\left\|\frac{e^{\frac{\alpha}{2}\zeta}}{\beta}\int_{\zeta}^{\infty}e^{-\frac{\alpha}{2}\sigma}d\sigma\right\|\|\eta_1-\eta_2\|\nonumber\\
&=\frac{4\e}{\alpha\beta}\|\eta_1-\eta_2\|.\label{S41L0E4}
\end{align}
On the other hand,
\begin{align}
\|\calF(0)\|&\le\left\|\frac{e^{\frac{\alpha}{2}\zeta}}{\beta}\int_{\zeta}^{\infty}e^{-\frac{\alpha}{2}\sigma}|g(0,\sigma)|d\sigma\right\|\nonumber\\
&\le\left\|\frac{e^{\frac{\alpha}{2}\zeta}}{\beta}\int_{\zeta}^{\infty}e^{-\frac{\alpha}{2}\sigma}m^2e^{-2m\sigma}d\sigma\right\|\nonumber\\
&=\frac{m^2}{\beta\left(\frac{\alpha}{2}+2m\right)}e^{-2m\zeta}\nonumber\\
&<\e,\label{S41L0E5}
\end{align}
provided that $\zeta>0$ is large. By (\ref{S41L0E4}) and (\ref{S41L0E5}) we see that if $\e>0$ is small, then
\begin{align}
\|\calF(\eta_1)\| &=\|\calF(\eta_1)-\calF(0)+\calF(0)\|\nonumber\\
&\le \|\calF(\eta_1)-\calF(0)\|+\|\calF(0)\|\nonumber\\
&\le\frac{4\e}{\alpha\beta}\|\eta_1\|+\e\nonumber\\
&<\bar{\eta}.\label{S41L0E6}
\end{align}
Because of (\ref{S41L0E4}) and (\ref{S41L0E6}), if $\e>0$ is small, then $\calF$ is a contraction mapping from $\calB$ into itself. Therefore, (\ref{S41L0E1}) has a unique solution.
\end{proof}

Hereafter, let $y^*(t)$ be the solution given by Lemma~\ref{S41L0}, and
\begin{equation}\label{sing}
u^*(s):=As^{-\theta}y^*(m^{-1}\log s).
\end{equation}
Then the following corollary is an immediate consequence of Lemmas~\ref{S41L0} and \ref{S41L1}.
\begin{corollary}\label{S41C1}
\begin{equation}\label{S41C1E1}
u^*(s)=As^{-\theta}\left( 1+o(1)\right)\ \ \textrm{as}\ \ s\downarrow 0.
\end{equation}
Here, $A=A(p,N)$.
\end{corollary}

\subsection{Positivity of the entire singular solution}
\begin{lemma}\label{S42L1}
The domain of the singular solution $y^*(t)$ obtained in Lemma~\ref{S41L0} can be extended to $t\in (-\infty,\infty)$, and $y^*(t)>0$ $(-\infty<t<\infty)$. Therefore, the domain of $u^*(s)$ can be extended to $s\in(0,\infty)$, and $u^*(s)>0$ $(0<s<\infty)$.
\end{lemma}

\begin{proof}
In the proof we denote $y^*$ by $y$ for simplicity.
We extend the domain of $y(t)$.
The equation (\ref{S4E1}) is equivalent to
\begin{equation}\label{S42E1}
\begin{cases}
y'=z, & \\
z'=-\alpha z+y-y^p+m^2e^{2mt}y. &
\end{cases}
\end{equation}
Then, the solution of (\ref{S4E1}) can be considered as the orbit $(y(t),z(t))$ of (\ref{S42E1}).
Let
\[
\tilde{E}(y,z):=\frac{z^2}{2}-\frac{\beta(t)}{2}y^2+\frac{y^{p+1}}{p+1},
\]
where $\beta(t):=1+m^2e^{2mt}$. Then
\begin{equation}\label{S42E3}
\frac{d}{dt}\tilde{E}(y(t),z(t))=-\alpha z^2-m^3e^{2mt}y^2\le 0.
\end{equation}
Since $(y(-\infty),z(-\infty))=(1,0)$, $\tilde{E}(y(-\infty),z(-\infty))=-1/2+1/(p+1)<0$.
Because of (\ref{S42E3}), $(y(t),z(t))\in \{(y,z);\ \tilde{E}(y,z)<0\}$. Since $-\beta(t)y^2/2+y^{p+1}/(p+1)<0$,
\begin{equation}\label{S42E4}
0<y(t)<\left(\frac{p+1}{2}\beta(t)\right)^{\frac{1}{p-1}}\ \ (-\infty<t<\infty).
\end{equation}
Therefore, $y(t)>0$ $(-\infty<t<\infty)$. Since $\tilde{E}(y,z)<0$, $z^2<\beta(t)y^2-2y^{p+1}/(p+1)$, hence
\begin{equation}\label{S42E5}
0<z(t)<\left(\frac{p-1}{p+1}\beta(t)^{1+\theta}\right)^{\frac{1}{2}}\ \ (-\infty<t<\infty).
\end{equation}
Because of (\ref{S42E4}) and (\ref{S42E5}), neither $y(t)$ nor $z(t)$ blows up. Hence, the solution $(y(t),z(t))$ of (\ref{S42E1}) can be extended to $t\in (-\infty,\infty)$.
Because of (\ref{sing}), the conclusion about $u^*(s)$ also holds.
\end{proof}

\subsection{Nonexistence of the entire singular solution}
\begin{lemma}\label{S42+L1}
Let $y^*(t)$ be as in Lemma~\ref{S41L0}. Then,
\begin{equation}\label{S42+L1E0}
(y^*)'(t)=\frac{m}{2(N-1)-3\theta}e^{2mt}+O(e^{2m(1+\nu)t})\quad (t\rightarrow -\infty),
\end{equation}
where $\nu=\min\{p-1,1\}$.
\end{lemma}

\begin{proof}
Let $\zeta:=-t$ and $\eta(\zeta)=y(t)-1$. By (\ref{S41L1E3}) and (\ref{S41L1E10}) we have
\begin{equation}\label{S42+L1E1}
|\varphi(\eta)|\le C_1e^{-2m(1+\nu)\zeta}.
\end{equation}
Because of (\ref{S41L1E3}) and (\ref{S42+L1E1}), we have
\begin{equation}\label{S42+L1E2}
g(\zeta)=m^2e^{-2m\zeta}+O(e^{-2m(1+\nu)\zeta}).
\end{equation}
There are three cases (\ref{S41L1E3+}) as in the proof of Lemma~\ref{S41L1}.
We consider only the case (1).

Differentiating (\ref{S41L1E4}) with respect to $\zeta$, we have
\begin{multline}\label{S42+L1E3}
\eta'(\zeta)=\frac{\alpha}{2\beta}e^{\frac{\alpha}{2}\zeta}\int_{\zeta}^{\infty}e^{-\frac{\alpha}{2}\sigma}\sin (\beta(\sigma-\zeta))g(\sigma)d\sigma\\
-e^{\frac{\alpha}{2}\zeta}\int_{\zeta}^{\infty}e^{-\frac{\alpha}{2}\sigma}\cos(\beta(\sigma-\zeta))g(\sigma)d\sigma.
\end{multline}
By (\ref{S42+L1E2}) and (\ref{S42+L1E3}) we have
\begin{align*}
\eta'(\zeta) &=\frac{\alpha m^2}{2\beta}e^{\frac{\alpha}{2}\zeta}\int_{\zeta}^{\infty}e^{-\left(\frac{\alpha}{2}+2m\right)\sigma}\sin (\beta(\sigma-\zeta))d\sigma\\
&\quad -m^2e^{\frac{\alpha}{2}\zeta}\int_{\zeta}^{\infty}e^{-\left(\frac{\alpha}{2}+2m\right)\sigma}\cos(\beta(\sigma-\zeta))d\sigma +O(e^{-2m(1+\nu)\zeta})\\
&=\frac{-2m^3}{\left(\frac{\alpha}{2}+2m\right)^2+\beta^2}e^{-2m\zeta}+O(e^{-2m(1+\nu)\zeta})\\
&=-\frac{m}{2(N-1)-3\theta}e^{-2m\zeta}+O(e^{-2m(1+\nu)\zeta}).
\end{align*}
Since $(y^*)'(t)=-\eta'(\zeta)$, we obtain (\ref{S42+L1E0}).
In the other cases (2) and (3) we also obtain (\ref{S42+L1E0}), using (\ref{S41L1E11}) and (\ref{S41L1E16}).
\end{proof}

\begin{corollary}\label{S42+C1}
Suppose that $p>p_S$.
Let $\delta>0$. Then, $u^*\in H^1(B_{\delta})$, and
\begin{equation}\label{S42+C1E1}
(u^*)'(s)=-\theta A s^{-\theta-1}(1+o(1))\quad (s\downarrow 0).
\end{equation}
\end{corollary}

\begin{proof}
Differentiating (\ref{sing}) with respect to $s$, we have
\begin{align*}
(u^*)'(s)&=-\theta As^{-\theta-1}y^*(\frac{1}{m}\log s)+As^{-\theta}(y^*)'(\frac{1}{m}\log s)\frac{1}{s}\\
&=-\theta As^{-\theta-1}(1+o(1))+As^{-\theta-1}o(s)\\
&=-\theta As^{-\theta-1}(1+o(1)),
\end{align*}
where we use Lemma~\ref{S42+L1}.
Then, we obtain (\ref{S42+C1E1}).
We show that
\[
\int_0^{\delta}\left\{(u^*)^2+((u^*)')^2\right\}s^{N-1}ds<\infty.
\]
We have
\begin{align*}
\left\{ u^*(s)^2+((u^*)'(s))^2\right\}s^{N-1}&\le\left( C_0s^{-2\theta}+C_1s^{-2\theta-2}\right)s^{N-1}\\
&=C_0s^{-2\theta+N-1}+C_1s^{-2\theta-2+N-1}.
\end{align*}
Since $p>p_S$, $-2\theta-2+N-1>-1$ and $-2\theta+N-1>-1$. Hence, $\int_0^{\delta}(C_0s^{-2\theta+N-1}+C_1s^{-2\theta-2+N-1})ds<\infty$, which indicates that $u^*\in H^1(B_{\delta})$.
\end{proof}

\begin{lemma}\label{S42+L2}
Suppose that $p>p_S$. The problem
\begin{equation}\label{S42+L3E0}
\begin{cases}
u''+\frac{N-1}{s}u'+f(u)=0, & 0<s<\infty,\\
u(s)=As^{-\theta}(1+o(1))\quad (s\downarrow 0),\\
u(\infty)=0,\\
u>0, & 0<s<\infty
\end{cases}
\end{equation}
does not admit a solution.
\end{lemma}
This lemma was proved by \cite{NS86}. See \cite{CCCT10} for details of nonexistence results.
However, we briefly prove the lemma.
\begin{proof}
Suppose that (\ref{S42+L3E0}) has a solution.
Because $u(s)=As^{-\theta}(1+o(1))$ $(s\downarrow 0)$ and the uniqueness of the solution $u^*(s)$ obtained by Lemma~\ref{S41L0}, the solution of (\ref{S42+L3E0}) is $u^*(s)$.
We denote $u^*$ by $u$ for simplicity.
Because $f'(0)=-1<0$, $u(s)$ and $u'(s)$ decay exponentially as $s\rightarrow\infty$. Thus, the four terms $R^{N-1}u(R)u'(R)$, $R^Nu'(R)^2$, $R^Nu(R)^2$, and $R^Nu(R)^{p+1}$ converge to zero as $R\rightarrow\infty$.
Because of (\ref{S41C1E1}) and (\ref{S42+C1E1}), we see that the four terms $\e^{N-1}u(\e)u'(\e)$, $\e^Nu'(\e)^2$, $\e^Nu(\e)^2$, and $\e^Nu(\e)^{p+1}$ converge to zero as $\e\downarrow 0$.
Taking the limit of (\ref{S22P1E1}) as $\e\downarrow 0$ and $R\rightarrow\infty$, we have
\[
\left(\frac{N}{2}-1-\frac{N}{p+1}\right)\int_0^{\infty}s^{N-1}(u')^2ds+\left(\frac{N}{2}-\frac{N}{p+1}\right)\int_0^{\infty}s^{N-1}u^2ds=0.
\]
Since $p>p_S$, $\frac{N}{2}-1-\frac{N}{p+1}>0$. The above equality implies that (\ref{S42+L3E0}) does not admit a solution.
\end{proof}

\subsection{Existence of critical points of $u^*(s)$}
The main technical lemma in this section is the following:
\begin{lemma}\label{S43L1}
Suppose that $p_S<p<p_{JL}$.
There is a sequence $\{ s_n^*\}_{n=0}^{\infty}$ $(0=s_0^*<s_1^*<s_2^*<\cdots\rightarrow \infty)$ such that $(u^*)'(s_n^*)=0$ $(n\ge 1)$,
\[
u^*(s)
\begin{cases}
<1 & (s\in\{ s_1^*,s_3^*,\cdots\}),\\
>1 & (s\in\{ s_2^*,s_4^*,\cdots\})
\end{cases}
\]
and
\[
(u^*)'(s)
\begin{cases}
<0 & (s\in (s_0^*,s_1^*)\cup(s_2^*,s_3^*)\cup\cdots),\\
>0 & (s\in (s_1^*,s_2^*)\cup(s_3^*,s_4^*)\cup\cdots).
\end{cases}
\]
\end{lemma}

\begin{proof}
First, we show that $u^*$ has the first positive critical point.
We use a contradiction argument. Suppose the contrary, i.e., 
\[
(u^*)'(s)<0\ \ (0<s<\infty).
\]
Since $u^*(s)>0$ $(0<s<\infty)$, $u^*(s)$ converges as $s\rightarrow\infty$.
Since $(u^*)'$ does not change sign, $(u^*)'(s)\rightarrow 0$ $(s\rightarrow\infty)$.
We see by (\ref{S1E7}) that $(u^*)''(s)$ converges as $s\rightarrow\infty$.
Since $u^*(s)$ converges, $(u^*)''(s)\rightarrow 0$ $(s\rightarrow\infty)$.
Therefore, $f(u^*(\infty))=0$, and $u^*(\infty)=0$ or $1$.
If $u^*(\infty)=1$, then by the same argument as in the proof of Lemma~\ref{BDDB} we see that $u^*$ should oscillate around $1$. This oscillation contradicts that $(u^*)'(s)<0$ $(0<s<\infty)$.
Thus, $u^*(\infty)=0$.
We see that $u^*$ is a solution of (\ref{S42+L3E0}), which contradicts Lemma~\ref{S42+L2}. 
Therefore, there is $s_1^*>0$ such that $(u^*)'(s)<0$ $(0<s<s_1^*)$ and $(u^*)'(s_1^*)=0$.

Second, we show that there exists $s_n^*$ ($n\ge 2$).
Let $v^*:=(u^*)'$, and let $E$ be as given in (\ref{S2E22++}).
Then
\[
\frac{d}{ds}E(u^*(s),v^*(s))=-\frac{N-1}{s}(v^*(s))^2<0.
\]
Therefore, $E(u^*(s),v^*(s))<E(u^*(s_1^*),0)<0$ $(s>s_1^*)$. This implies that
\begin{equation}\label{S43L1E1}
u^*(s)>u^*(s_1^*)\quad (s>s_1^*).
\end{equation}
Let $w(s):=u^*(s)-1$. Then $w$ satisfies (\ref{BDDBE2+}).
Because of (\ref{S43L1E1}), $\{(w+1)^p-(w+1)\}/w>\{(u^*(s_1^*))^p-u^*(s_1^*)\}/(u^*(s_1^*)-1)>0$.
Sturm's oscillation theorem says that $w$ oscillates around $0$ infinitely many times as $s\rightarrow\infty$.
Consequently, $u^*$ has infinitely many critical points.
Because $u^*$ satisfies the ODE, we easily see that $u^*$ does not have a critical point on $\{u^*=1\}$ and that if $u^*$ has a critical point on $\{ u^*>1\}$ (resp. $\{ u^*<1\}$), then it is a local maximum point (resp. local minimum point).
Thus, the conclusion of the lemma holds.
\end{proof}

\subsection{Proof of Theorem~\ref{Th1}}
Since the convergence property in Theorem~\ref{Th0} is proved in Section~5, the proof of Theorem~\ref{Th0} is postponed until Section~5.
Here we prove only Theorem~\ref{Th1}.
\begin{proof}[Proof of Theorem~\ref{Th1}]
Let $\lambda^*_n:=(s_n^*)^2$. Then, $(\lambda_n^*,u^*(s))$ is a singular solution of (\ref{S1E7}) which satisfies (\ref{S41C1E1}).
We define $U_n^*(r):=u^*(\sqrt{\lambda_n^*}r)$.
Then by Corollary~\ref{S41C1} and Lemmas \ref{S42L1} and \ref{S43L1} we see that $U_n^*$ satisfies (i), (ii), and (iii).
It follows from Corollary~\ref{S42+C1} that $U^*_n\in H^1(B)$.
The uniqueness of $(\lambda_n^*,U_n^*)$ follows from the uniqueness of the solution of (\ref{S41L0E0}) which was shown in Lemma~\ref{S41L0}.
\end{proof}

\section{Convergence to the singular solution as $\gamma\rightarrow\infty$}
Let $u(s,\gamma)$ denote the solution of (\ref{S3E2}). Note that $u(0,\gamma)=\gamma$.
Let $y(t)=A(p,N)^{-1}s^{\theta}u(s,\gamma)$ and $t=m^{-1}\log s$.
Then $y(t)$ satisfies (\ref{S4E1}).
Let $n\in\N$ be fixed.
Let $\lambda_n^*$ be the number given in Theorem~\ref{Th1}, and let $y^*(t)$ be the singular solution given in Lemma~\ref{S41L0}.
Then by Lemma~\ref{S42L1} we see that $y(t)$ can be defined in $\R$.
We set $t_n^*:=m^{-1}\log\sqrt{\lambda_{n+1}^*}$.
\subsection{Convergence to $u^*$}
Our goal in this section is to prove the following:
\begin{theorem}\label{S5T1}
Let $n\ge 1 $ be fixed.
Let $r_n^*:=e^{mt_n^*}$.
Suppose that $\{\gamma_j\}_{j=1}^{\infty}$ is an arbitrary sequence diverging to $\infty$. Let $u^*(s)$ be the singular solution defined by (\ref{sing}), and let $u_j(s):=u(s,\gamma_j)$ be the solution of (\ref{S3E2}).
Then, as $j\rightarrow\infty$,
\[
u_j\rightarrow u^*,\quad u'_j\rightarrow (u^*)',\quad\textrm{and}\quad u''_j\rightarrow (u^*)''\quad\textrm{in}\quad C^0_{loc}(0,r^*_n].
\]
\end{theorem}
We define
\[
\tau(\gamma):=\frac{1}{m\theta}(\log\gamma-\log A).
\]
\begin{lemma}\label{S5L1}
Suppose that $\gamma>0$ is large.
Then, $y(t)$ satisfies
\[
y(t,\gamma)\le e^{m\theta(t+\tau(\gamma))}\quad (-\infty<t\le t_n^*).
\]
\end{lemma}

\begin{proof}
Let $u(s):=u(s,\gamma)$ and $v(s):=u'(s)$.
Since $(u(s),v(s))$ satisfies (\ref{BDDBE3}),
\begin{equation}\label{S5L1E1}
-\frac{u(s)^2}{2}+\frac{u(s)^{p+1}}{p+1}\le E(u(s),v(s))\le E(\gamma,0)=-\frac{\gamma^2}{2}+\frac{\gamma^{p+1}}{p+1}\quad (s\ge 0).
\end{equation}
Since $\gamma>0$ is large, (\ref{S5L1E1}) tells us that $u(s)\le\gamma$ ($s\ge 0$). By the definition of $y(t)$ we have
\[
y(t)=A^{-1}s^{\theta}u(s)\le A^{-1}e^{m\theta t}\gamma=e^{m\theta(t+\tau(\gamma))}\quad (-\infty<t\le t_n^*).
\]
\end{proof}

We define
\[
\xi:=t+\tau(\gamma)\quad\textrm{and}\quad w(\xi,\gamma):=y(t,\gamma).
\]
Then $w(\xi,\gamma)$ is a solution of the problem
\begin{equation}\label{S5E1}
\begin{cases}
w''+\alpha w'-w+w^p-m^2e^{2m(\xi-\tau(\gamma))}w=0, & -\infty<\xi<t_n^*+\tau(\gamma),\\
w>0, & -\infty<\xi<t_n^*+\tau(\gamma),\\
e^{-m\theta\xi}w(\xi)\rightarrow 1\ \ (\xi\rightarrow -\infty).
\end{cases}
\end{equation}
Because of Lemma~\ref{S5L1},
\begin{equation}\label{S5E3}
w(\xi,\gamma)\le e^{m\theta\xi}\quad (-\infty<\xi\le t_n^*+\tau(\gamma)).
\end{equation}
Using the a priori estimate (\ref{S5E3}), we prove the following uniform convergence:
\begin{lemma}\label{S5L2}
Let $\bar{w}(\xi)$ be a solution of
\begin{equation}\label{S5L2E0}
\begin{cases}
w''+\alpha w'-w+w^p=0, & -\infty<\xi<t_n^*+\tau(\gamma),\\
w>0, & -\infty<\xi<t_n^*+\tau(\gamma),\\
e^{-m\theta\xi}w(\xi)\rightarrow 1\ \ (\xi\rightarrow -\infty).
\end{cases}
\end{equation}
Then, for each fixed $\bar{\xi}\in\R$,
\[
w(\xi,\gamma)\rightarrow\bar{w}(\xi)\quad\textrm{and}\quad w'(\xi,\gamma)\rightarrow\bar{w}'(\xi)\quad\textrm{in}\quad C^0(-\infty,\bar{\xi})
\]
as $\gamma\rightarrow\infty$.
\end{lemma}

\begin{proof}
Let $\e>0$ be small. Since $\bar{w}(\xi)$ is a solution of (\ref{S5L2E0}), 
\begin{equation}\label{S5L2E1}
\bar{w}(\xi)\rightarrow 0\quad (\xi\rightarrow -\infty).
\end{equation}
Because of (\ref{S5E3}) and (\ref{S5L2E1}), there is $\underline{\xi}(<\bar{\xi})$ such that
\[
|w(\xi,\gamma)|\le\frac{\e}{2}\quad\textrm{and}\quad |\bar{w}(\xi)|\le\frac{\e}{2}\quad (-\infty<\xi\le\underline{\xi}).
\]
Then,
\begin{align}
|w(\xi,\gamma)-\bar{w}(\xi)|&\le |w(\xi,\gamma)|+|\bar{w}(\xi)|\nonumber\\
&\le \frac{\e}{2}+\frac{\e}{2}\nonumber\\
&=\e\quad (-\infty<\xi<\underline{\xi}).\label{S5L2E2}
\end{align}
We show that
\begin{equation}\label{S5L2E3}
|w(\xi,\gamma)-\bar{w}(\xi)|<\e\quad\textrm{in}\quad [\underline{\xi},\bar{\xi}].
\end{equation}
Since $y'(t)=A^{-1}ms^{\theta}(\theta u(s)+su'(s))$, $y'(t,\gamma)\rightarrow 0$ $(t\rightarrow -\infty)$. Hence, $w'(\xi,\gamma)\rightarrow 0$ $(\xi\rightarrow -\infty)$.
Integrating the equation in (\ref{S5E1}) over $(-\infty,\xi]$, we have
\begin{equation}\label{S5L2E4}
w'(\xi)+\alpha w(\xi)=\int_{-\infty}^{\xi}\left( w-w^p+m^2e^{2m(\eta-\tau(\gamma))}w\right)d\eta,
\end{equation}
where we use $w'(-\infty)=0$ and $w(-\infty)=0$. There is $C_0>0$ independent of large $\gamma>0$ such that
\[
\int_{-\infty}^{\xi}\left| w-w^p+m^2e^{2m(\eta-\tau(\gamma))}w\right|d\eta<C_0\quad (\xi\le\bar{\xi}),
\]
because of (\ref{S5E3}). Then $|w'(\xi)|\le |\alpha w(\xi)|+C_0$, which indicates that $w$ is equicontinuous in $[\underline{\xi},\bar{\xi}]$. It follows from (\ref{S5E3}) that $w(\xi,\gamma)$ is uniformly bounded in $[\underline{\xi},\bar{\xi}]$. By the Arzel\`a-Ascoli theorem we see that $w(\xi,\gamma)$ converges to a certain function in $C^0[\underline{\xi},\bar{\xi}]$ which satisfies the equation in (\ref{S5L2E0}).
This function is a solution of (\ref{S5L2E0}). Using the same argument as in the proof of Lemma~\ref{S41L0}, we see that the solution of (\ref{S5L2E0}) is unique, hence the limit function is $\bar{w}(\xi)$. Therefore, (\ref{S5L2E3}) holds for large $\gamma>0$. The estimates (\ref{S5L2E2}) and (\ref{S5L2E3}) indicate that $w(\xi,\gamma)\rightarrow\bar{w}(\xi)$ in $C^0(-\infty,\bar{\xi}]$ as $\gamma\rightarrow\infty$. In particular, this convergence is uniform.

Next, we show that
\begin{equation}\label{S5L2E5}
w'(\xi,\gamma)\rightarrow\bar{w}'(\xi)\quad\textrm{in}\quad C^0(-\infty,\bar{\xi}]\quad\textrm{as}\quad\gamma\rightarrow\infty.
\end{equation}
Because of (\ref{S5E3}), we see by (\ref{S5L2E4}) that $|w'(\xi,\gamma)|$ is uniformly bounded in each bounded interval. Since $w''=-\alpha w'+w-w^p+m^2e^{2m(\xi-\tau(\gamma))}w$, $|w''(\xi,\gamma)|$ is also uniformly bounded in each bounded interval.
By the Arzel\`a-Ascoli theorem we see that $w'(\xi,\gamma)$ uniformly converges in a bounded interval as $\gamma\rightarrow\infty$.
Since the solution of (\ref{S5L2E0}) is unique, the limit function is $\bar{w}'(\xi)$. Hence
\begin{equation}\label{S5L2E6}
w'(\xi,\gamma)\rightarrow\bar{w}'(\xi)\quad\textrm{in}\quad C^0_{loc}(-\infty,\bar{\xi}].
\end{equation}
By the same argument as in the case $w'(-\infty)$ we can show that $\bar{w}'(-\infty)=0$. Therefore, there is $\xi_0$ such that
\begin{equation}\label{S5L2E7}
|w'(\xi,\gamma)-\bar{w}'(\xi)|<\e\quad (-\infty<\xi<\xi_0).
\end{equation}
Because of (\ref{S5L2E6}) and (\ref{S5L2E7}), (\ref{S5L2E5}) holds. The proof is complete.
\end{proof}

Since $(\bar{w},\bar{w}')\rightarrow (1,0)$ ($\xi\rightarrow\infty$), it follows from Lemma~\ref{S5L2} that $(y,z)$ approaches $(1,0)$ as $\gamma\rightarrow\infty$ along $t=\bar{\xi}-\tau(\gamma)$ provided that $\bar{\xi}$ is large enough.
In the next lemma we show that once near $(1,0)$, the orbit $(y,z)$ remains near $(1,0)$ until $t$ reaches some $T$, which is negatively large, but independent of $\gamma$.
We set
\[
\Gamma_{\rho}:=\{(y,z);\ E(y,z)<E(1,0)+\rho\},\ \ \rho>0,
\]
where $E$ is defined by (\ref{S2E22++}).
\begin{lemma}\label{S5L3}
For every $\e>0$, there exists $T_{\e}(<t^*)$ such that the following holds: If there is $t_0(<T_{\e})$ such that $(y(t_0),z(t_0))\in\Gamma_{\e}$, then
\[
(y(t),z(t))\in\Gamma_{2\e}\quad (t_0\le t\le T_{\e}).
\]
\end{lemma}

\begin{proof}
Let $E(t):=E(y(t),z(t))$. We have
\[
E'(t)=-\alpha z^2+m^2e^{2mt}yz\le\frac{m^4}{4\alpha}e^{4mt}y^2.
\]
Hence, for $t>t_0$,
\[
E(t)\le E(t_0)+\frac{m^4}{4\alpha}\int_{t_0}^te^{4m\zeta}y(\zeta)^2d\zeta.
\]
However, there is $K>0$ such that
\[
E(y,z)\ge -K+y^2.
\]
Therefore,
\begin{equation}\label{S5L3E4}
E(t)\le E(t_0)+\frac{m^4}{4\alpha}\left(\frac{K}{4m}e^{4mt}+\int_{t_0}^te^{4m\zeta}E(\zeta)d\zeta\right).
\end{equation}
We choose $T_{\e}$ so that
\[
\frac{m^4}{4\alpha}\left\{\frac{K}{4m}+\frac{1}{4m}\left( E(1,0)+2\e\right)\right\} e^{4mT_{\e}}<\e.
\]
Let $t_0(<T_{\e})$ and define
\[
t_1:=\sup\{ t>t_0;\ E(t)<E(1,0)+2\e\}.
\]
We show by contradiction that $t_1>T_{\e}$. Suppose the contrary, i.e., $t_1\le T_{\e}$. Then by (\ref{S5L3E4}) and the definition of $t_1$ we have
\begin{align*}
E(t_1) &\le E(t_0)+\frac{m^4}{4\alpha}\left(\frac{K}{4m}e^{4mt}+\int_{t_0}^te^{4m\zeta}E(\zeta)d\zeta\right)\\
&<E(t_0)+\e\\
&\le E(1,0)+\e+\e\\
&=E(1,0)+2\e,
\end{align*}
which contradicts the definition of $t_1$. Hence, $t_1>T_{\e}$, which indicates that $(y,z)\in\Gamma_{2\e}$ ($t_0\le t\le T_{\e}$).
\end{proof}

\begin{lemma}\label{S5L4}
Suppose that $\gamma>0$ is large. There exists a constant $M>0$ independent of $\gamma>0$ such that
\[
|y(t)|+|z(t)|\le M\quad (t\le t_n^*).
\]
\end{lemma}

\begin{proof}
By the same argument as in the proof of Lemma~\ref{S5L2} we show that $y'(-\infty)=0$. Let $\tilde{E}:=\tilde{E}(y(t),z(t))$. Since $\tilde{E}(t)\le\tilde{E}(-\infty)=0$, hence
\begin{equation}\label{S5L4E1}
\frac{z(t)^2}{2}-\frac{\beta(t)}{2}y(t)^2+\frac{y(t)^{p+1}}{p+1}\le 0\quad\textrm{for}\quad t\in\R.
\end{equation}
When $t\in (-\infty,t_n^*]$, $\beta(t)$ is bounded, hence there is $K>0$ such that
\begin{equation}\label{S5L4E2}
-\frac{\beta(t)}{2}y^2+\frac{y^{p+1}}{p+1}\ge -\frac{K}{2}+\frac{y^2}{2}\quad (-\infty<t\le t_n^*).
\end{equation}
We see by (\ref{S5L4E1}) and (\ref{S5L4E2}) that for $t\in (-\infty,t_n^*]$,
\[
\frac{z(t)^2}{2}+\frac{y(t)^2}{2}-\frac{K}{2}\le\frac{z(t)^2}{2}-\frac{\beta(t)}{2}y(t)^2+\frac{y(t)^{p+1}}{p+1}\le 0,
\]
hence $y(t)^2+z(t)^2\le K$. The conclusion holds.
\end{proof}

\begin{proof}[Proof of Theorem~\ref{S5T1}]
Let $(y_j,z_j)$ be a solution of (\ref{S42E1}) such that $e^{-m\theta t}y(t)\rightarrow\frac{\gamma_j}{A}$ ($t\rightarrow -\infty$).
By Lemma~\ref{S5L4}, the sequence $\{(y_j,z_j)\}$ is uniformly bounded in $(-\infty,t_n^*]$. Since $(y_j,z_j)$ is a solution of (\ref{S42E1}), the sequence $\{(y_j,z_j)\}$ is uniformly bounded in $(C^2(-\infty,t_n^*])^2$.
By the Arzel\`a-Ascoli theorem we can extract a subsequence, which is still denoted by $\{(y_j,z_j)\}$, converges in $(C^1(I))^2$ to some function $(\bar{y},\bar{z})$, where $I$ is an arbitrary compact subset of $(-\infty,t_n^*]$.
By taking the limit, we find that $(\bar{y},\bar{z})$ satisfies (\ref{S42E1}).

It remains to show that
\begin{equation}\label{S5T1E1}
\bar{y}(t)\rightarrow 1\quad\textrm{as}\quad t\rightarrow -\infty,
\end{equation}
for then $\bar{y}$ is the solution of the problem
\[
\begin{cases}
y''+\alpha y'-y+y^p-m^2e^{2mt}y=0, & -\infty<t<t_n^*,\\
y(-\infty)=1.
\end{cases}
\]
We have shown in Lemma~\ref{S41L0} that this problem has a unique solution $y^*(t)$. Thus, $\bar{y}(t)\equiv y^*(t)$ ($-\infty<t\le t_n^*$). Because of this uniqueness, the entire sequence $\{y_j\}$ converges to $y^*$. Thus, as $j\rightarrow\infty$,
\[
y_j\rightarrow y^*\quad\textrm{and}\quad
y'_j\rightarrow (y^*)'
\quad\textrm{in}\quad C^0_{loc}(-\infty,t_n^*].
\]
Hence, the first two convergences of the statement of the theorem are obtained.
The third convergence holds, since $u_j$ satisfies (\ref{S3E2}).

We show (\ref{S5T1E1}) by contradiction argument. Suppose that (\ref{S5T1E1}) does not hold. Then there exists a sequence $\{t_k\}$ such that $t_k\rightarrow -\infty$ as $k\rightarrow\infty$, and a constant $\delta>0$ so that
\begin{equation}\label{S5T1E2}
(\bar{y}(t_k),\bar{z}(t_k))\not\in\Gamma_{\delta}\quad\textrm{for all}\quad k\ge 1.
\end{equation}
Let $\e:=\frac{\delta}{4}$. Then by Lemma~\ref{S5L2}, there exist $\bar{\xi}\in\R$ and $j_0\in\N$ such that
\[
(y_j(\bar{\xi}-\tau(\gamma_j)),z_j(\bar{\xi}-\tau(\gamma_j)))\in\Gamma_{\e}\quad\textrm{if}\quad j>j_0.
\]
By Lemma~\ref{S5L3}, this implies that if $j>j_0$, then
\[
(y_j(t),z_j(t))\in\Gamma_{2\e}\subset\Gamma_{\delta}\quad\textrm{in}\quad (\bar{\xi}-\tau(\gamma_j),T_{\e}),
\]
where $T_{\e}$ depends only on $\e$. By choosing $j$ large enough, the interval $(\bar{\xi}-\tau(\gamma_j),T_{\e})$ can be made to include an element of the sequence $\{t_k\}$.
We obtain a contradiction, because of (\ref{S5T1E2}).
\end{proof}
We are in a position to prove Theorem~\ref{Th0}.
\begin{proof}[Proof of Theorem~\ref{Th0}]
The existence of $u^*(s)$ near $s=0$ follows from Lemma~\ref{S41L0}.
Lemma~\ref{S42L1} says that $u^*$ is positive and it is defined in $s\in (0,\infty)$.
It follows from Lemma~\ref{S43L1} that $u^*$ oscillates around $1$.
The convergence property follows from Theorem~\ref{S5T1}.
\end{proof}

\subsection{Convergence to $\lambda_n^*$}
Using Theorem~\ref{S5T1}, we  prove the following:
\begin{theorem}\label{S5T2}
Let $\{\lambda_n(\gamma)\}_{\gamma>1}$ be as in Subsection~3.3, and let $\lambda_n^*$ be as in Theorem~\ref{Th1}. Then
\[
\lambda_n(\gamma)\rightarrow\lambda_n^*\quad\textrm{as}\quad\gamma\rightarrow\infty.
\]
\end{theorem}

\begin{proof}
Let $\tilde{s}_1^*$ be the first positive zero of $u^*(s)-\left(\frac{p+1}{2}\right)^{\frac{1}{p-1}}$, and let $I_1:=(0,\frac{\tilde{s}_1^*}{2})$.
Let $s_1$ be the first positive critical point of $u(s,\gamma)$. We show by contradiction that if $\gamma>0$ is large,then $u(s,\gamma)$ does not have a critical point in $I_1$. Suppose that $s_1\in I_1$. Then, $s_1$ should be a local minimum point of $u(s,\gamma)$, hence by (\ref{S3E2}), $u(s_1,\gamma)<1$. Then the corresponding orbit in the $(u,u')$-plane is in $\{E\le 0\}$ at $s=s_1$, and it cannot go out when $s>s_1$. This contradicts Theorem~\ref{S5T1}, because $u(s,\gamma)\rightarrow u^*(s)$ in $C^0[\frac{\tilde{s}_1^*}{2},\tilde{s}_1^*]$. We have shown that $u$ does not have a critical point in $I_1$.

Let $s_n^*$ be the $n$-th positive critical point of $u^*(s)$.
Let $\delta>0$ be small.
Because of Theorem~\ref{S5T1}, $u(s,\gamma)\rightarrow u^*(s)$ in $C^2[\frac{\tilde{s}_1^*}{2},s_n^*+\delta]$ as $\gamma\rightarrow\infty$.
Since each zero of $(u^*)'$ is simple, $s_n\rightarrow s_n^*$ ($\gamma\rightarrow\infty$), where $s_n$ is the $n$-th positive critical point of $u(s,\gamma)$.
Since $\sqrt{\lambda_n(\gamma)}=s_n$ and $\sqrt{\lambda_n^*}=s_n^*$, we obtain the conclusion.
\end{proof}

\section{Oscillation of $\lambda_n(\gamma)$}
Let $u(s,\gamma)$ be the solution of (\ref{S3E2}), and let $(\lambda_n^*,u^*)$ be a singular solution of (\ref{S1E7}). Let $\{\lambda_n(\gamma)\}_{\gamma>1}$ be as given in Subsection~3.3.
The goal of this section is to prove the following:
\begin{theorem}\label{S6T1}
Suppose that $p_S<p<p_{JL}$.
$\lambda_n(\gamma)$ oscillates around $\lambda_n^*$ infinitely many times as $\gamma\rightarrow\infty$.
\end{theorem}

\begin{proof}
Let $I:=(0,\sqrt{\lambda_n^*}]$.
We define $\tilde{u}(\rho,\gamma):=u(s,\gamma)/\gamma$, $\tilde{u}^*(\rho,\gamma):=u^*(s)/\gamma$, and $\rho:=\gamma^{\frac{p-1}{2}}s$.
We show that
\begin{multline}\label{S6T1E1}
\textrm{for each $M\ge 1$, there are $\rho_0\in [0,1]$ and $\gamma_0>0$ such that
}\\
\calZ_{(0,\rho_0)}[\tilde{u}^*(\,\cdot\,,\gamma)-\tilde{u}(\,\cdot\,,\gamma)]\ge M\quad\textrm{for}\quad\gamma\ge\gamma_0.
\end{multline}
The function $\tilde{u}(\rho)$ satisfies (\ref{S43L1E9}).
Since the energy
\[
\tilde{G}(\tilde{u},\tilde{u}')=\frac{(\tilde{u}')^2}{2}-\frac{\tilde{u}^2}{2\gamma^{p-1}}+\frac{\tilde{u}^{p+1}}{p+1}
\]
decreases, $\tilde{G}(\tilde{u},\tilde{u}')\le \tilde{G}(1,0)<0$, hence $\tilde{u}\le \left(\frac{p+1}{2\gamma^{p-1}}\right)^{\frac{1}{p-1}}$.
Since $\tilde{u}/\gamma^{p-1}\rightarrow 0$ as $\gamma\rightarrow\infty$, 
\begin{equation}\label{S43L1E9_2}
\tilde{u}(\rho,\gamma)\rightarrow \bar{u}(\rho,1)\ \ \textrm{in}\ \ C^2_{loc}(0,\infty)\cap C^0_{loc}[0,\infty)\ \ (\gamma\rightarrow\infty),
\end{equation}
where $\bar{u}(\,\cdot\,,1)$ is the solution of (\ref{S1E7+}) with $\gamma=1$.
We apply the same change of variables to the singular solution $u^*(s)$.
We have
\begin{align}
\tilde{u}^*(\rho,\gamma)&:=\frac{1}{\gamma}A\left(\frac{\rho}{\gamma^{\frac{p-1}{2}}}\right)^{-\theta}(1+o(1))\quad (s\downarrow 0)\nonumber\\
&=A\rho^{-\theta}(1+o(1))\quad ({\rho}{\gamma^{-\frac{p-1}{2}}}\downarrow 0).\label{S43L1E10}
\end{align}
If $\rho$ is in a compact interval in $\R_+$, then $\rho\gamma^{-\frac{p-1}{2}}$ uniformly converges to $0$ on the interval as $\gamma\rightarrow\infty$.
Therefore, (\ref{S43L1E10}) indicates that 
\begin{equation}\label{S43L1E11}
\tilde{u}^*(\rho,\gamma)\rightarrow\bar{u}^*(\rho)\ \ \textrm{in}\ \ C^2_{loc}(0,\infty)\ \ (\gamma\rightarrow\infty),
\end{equation}
where $\bar{u}^*$ is given by (\ref{singent}) which is a singular solution of the equation in (\ref{S1E7+}).
Because of Proposition~\ref{S2P1}, $\calZ_{[0,\infty)}[\bar{u}^*(\,\cdot\,)-\bar{u}(\,\cdot\,,1)]=\infty$. Therefore, (\ref{S43L1E9_2}) and (\ref{S43L1E11}) indicate that (\ref{S6T1E1}) holds.

Since $0<\rho<\rho_0$ is equivalent to $0<s<\rho_0\gamma^{-\frac{p-1}{2}}$ and $\rho_0\gamma^{-\frac{p-1}{2}}\downarrow 0$ $(\gamma\rightarrow\infty)$, (\ref{S6T1E1}) implies that 
\begin{equation}\label{S6T1E2}
\calZ_I[u^*(\,\cdot\,)-u(\,\cdot\,,\gamma)]\rightarrow\infty\quad (\gamma\rightarrow\infty).
\end{equation}
Since each zero of $u^*(\,\cdot\,)-u(\,\cdot\,,\gamma)$ is simple, the number of the zeros in $I$ is preserved provided that another zero does not enter $I$ from $\partial I$. Since $u^*(0)-u(0,\gamma)=\infty$, a zero cannot enter $I$ from $s=0$.
Since (\ref{S6T1E2}) says that the number of the zeros increases, a simple zero enter $I$ from $s=\sqrt{\lambda_n^*}$ infinitely many times as $\gamma\rightarrow\infty$.
Therefore, there exists a sequence of large numbers $\{\gamma_j\}_{j=1}^{\infty}$ $(\gamma_1<\gamma_2<\cdots\rightarrow \infty)$ such that $u(\sqrt{\lambda_n^*},\gamma_j)=u^*(\sqrt{\lambda_n^*})$ and
\[
u_s(\sqrt{\lambda_n^*},\gamma_j)
\begin{cases}
<0, & (j\in\{1,3,5,\cdots\}),\\
>0, & (j\in\{2,4,6,\cdots\}).
\end{cases}
\]
Note that $u_s(\sqrt{\lambda_n^*},\gamma_j)\neq 0$, because of the simplicity of the zero.
We consider only the case where $n$ is odd, since the proof of the case where $n$ is even is similar to this case.
In the proof of Theorem~\ref{S5T2} we already saw that $u(s,\gamma_j)$ does not have a critical point in $(0,\frac{\tilde{s}_1}{2})$.
$\sqrt{\lambda_n^*}$ is a local minimum point of $u^*$.
Because of Theorem~\ref{S5T1}, $u(s,\gamma_j)$ is close to $u^*(s)$ in $C^2[\frac{\tilde{s}^*_1}{2},\sqrt{\lambda_n^*}]$.
We easily see that the $n$-th positive critical point of $u(s,\gamma)$, which is $\sqrt{\lambda_n(\gamma_j)}$, is close to $\sqrt{\lambda_n^*}$.
We consider the case where $u_s(\sqrt{\lambda_n^*},\gamma_j)<0$.
We show that $\lambda_n(\gamma_j)>\lambda_n^*$.
Suppose the contrary, i.e., $\lambda_n(\gamma_j)<\lambda_n^*$.
Since $u_s(\sqrt{\lambda_n^*},\gamma_j)<0$ and $s=\sqrt{\lambda_n^*}$ is a local minimum point of $u(\sqrt{\lambda_n(\gamma_j)},\gamma_j)$, there is a local maximum point $s^*\in(\sqrt{\lambda_n(\gamma_j)},\sqrt{\lambda_n^*)}$.
Since $s^*$ is a local maximum point, $u(s^*,\gamma_j)>1$ which contradicts Theorem~\ref{S5T1}. Thus, $\lambda_n(\gamma_j)>\lambda_n^*$.

When $u_s(\sqrt{\lambda_n^*},\gamma_j)>0$, we suppose the contrary, i.e., $\lambda_n(\gamma_j)>\lambda_n^*$. Since $\sqrt{\lambda_n(\gamma_j)}$ is a local minimum point, there is a local maximum point $s^*\in (\sqrt{\lambda_n(\gamma_j)},\sqrt{\lambda_n^*})$ and $u(s^*,\gamma_j)>1$.
We obtain a contradiction, because of Theorem~\ref{S5T1}.
Thus, $\lambda_n(\gamma_j)<\lambda_n^*$.

We have proved (\ref{S1E8}), which implies that $\lambda_n(\gamma)$ oscillates around $\lambda_n^*$ infinitely many times as $\gamma\rightarrow\infty$.
\end{proof}

\section{Boundary concentrating solution}
We construct a boundary concentrating solution which is radially increasing when $\lambda$ is large.
We show that this solution belongs to $\calC_1$.
We use different notation only in this section.

Let us compare Theorem~\ref{S77C1}, which is the main result of this section, with the results of Ambrosetti-Maclchiodi-Ni~\cite{AMN03,AMN04}.
In \cite{AMN03,AMN04} authors constructed various concentrating solutions of
\begin{equation}\label{S7E0}
\begin{cases}
\e^2\Delta u-V(|x|)u+u^p=0 & \textrm{in}\ \ \Omega\\
u>0 & \textrm{in}\ \ \Omega
\end{cases}
\end{equation}
under certain conditions on $V$, where $\Omega$ is the unit ball, annulus $\{ a<|x|<1\}$, or $\R^N$.
When $\Omega$ is the ball or annulus, the Dirichlet or Neumann boundary condition is imposed.
Among other things, they have shown that the Neumann problem on $B$ has a solution $u_{\e}$ of (\ref{S7E0}) concentrating on $|x|=r_{\e}$, where $r_{\e}$ is a local maximum for $u_{\e}$ such that $1-r_{\e}\sim\e\log\e$.
They used a finite dimensional reduction of the energy functional which works for general concentrating solutions.
Their solution is different from ours, since our solution concentrates on the boundary.
However, their method is applicable to our case.
In this paper, we use another method.
We find a solution near an approximate solution, using the contraction mapping theorem.

\subsection{One-dimensional problem}
We consider the one-dimensional problem
\begin{equation}\label{S71E1}
\begin{cases}
\e^2w''+f(w)=0,& 0<x<1,\\
w'(0)=w'(1)=0.
\end{cases}
\end{equation}
Let $\bar{w}(y):=w(x)$ and $y:=x/\e$.
The function $\bar{w}(y)$ satisfies
\begin{equation}\label{S71E2}
\begin{cases}
\bar{w}''+f(\bar{w})=0, & 0<y<d_{\e},\\
\bar{w}'(0)=\bar{w}'(d_{\e})=0,
\end{cases}
\end{equation}
where $d_{\e}:=1/\e$.

For $p>1$ the system of equations for $(\bar{w},\bar{z})$ $(\bar{z}:=\bar{w}')$ in the phase plane
\[
\begin{cases}
\bar{w}'=\bar{z}\\
\bar{z}'=-f(\bar{w})
\end{cases}
\]
has a saddle point at $(0,0)$ and a center $(1,0)$.
There is a unique homoclinic solution around the center connecting the saddle to itself.
This homoclinic solution can be written explicitly as
\[
w^*(y):=\left(\frac{p+1}{2}\right)^{\frac{1}{p-1}}\left(\cosh\left(\frac{p-1}{2}y\right)\right)^{-\frac{2}{p-1}}.
\]
From the phase portrait for $(\bar{w},\bar{z})$ it is clear that all the orbits on $\{\bar{w}>0\}$ satisfying the Neumann boundary conditions are inside the region enclosed by the homoclinic orbit and that every orbit in this region is periodic one.
Hence, $\bar{w}$ is a solution of (\ref{S71E2}) if and only if an integral multiple of its half period is equal to the interval length $d_{\e}$.
Now, we find a increasing solution.
Let $\bar{w}(y)$ be a increasing solution that has maximum $\alpha$ and minimum $\beta$.
Then,
\[
0<\beta<1<\alpha<\bar{\alpha}:=\left(\frac{p+1}{2}\right)^{\frac{1}{p-1}}.
\]
Multiplying (\ref{S71E2}) by $\bar{w}_y$ and integrating it, we have
\[
\bar{w}_y^2=F(\bar{w})-F(\beta),\qquad F(\alpha)=F(\beta),\quad\textrm{and}\quad F(\bar{w})=\bar{w}^2-\frac{2}{p+1}\bar{w}^{p+1}.
\]
The half period is given by the integral
\begin{equation}\label{TM}
T(\alpha):=\int_{\beta}^{\alpha}\frac{d\bar{w}}{\sqrt{F(\bar{w})-F(\beta)}}.
\end{equation}
Thus, $\bar{w}$ is a increasing solution of (\ref{S71E2}) if and only if $T(\alpha)=d_{\e}$.
The integral (\ref{TM}) was studied by De Groen and Karadzhov~\cite{DK02}.
Among other things, they obtained
\begin{proposition}\label{S71P1}
There is a small $\e_0>0$ such that the problem (\ref{S71E1}) has a smooth curve of increasing solutions $\{ w(x;\e)\}_{0<\e<\e_0}$, which can be described as a graph of $\e$, satisfying the following:
For $\delta>0$, there exists $\e_1>0$ such that, if $0<\e<\e_1$, then
\begin{equation}\label{S71P1E0}
|\tilde{w}(s;\e)-w^*(s)|<\delta\ \ \textrm{for}\ \ s\in [0,d_{\e}],
\end{equation}
where $\tilde{w}(s;\e):=w(x;\e)$ $(x=1-\e s)$.
Moreover, the first two eigenvalues of the eigenvalue problem
\[
\e^2\phi''+f'(w)\phi=\kappa\phi\ \ \textrm{in}\ \ (0,1),\qquad\phi'=0\ \ \textrm{at}\ \ x=0,1
\]
are
\[
\begin{split}
\kappa_0(\e)&=\frac{(p-1)(p+3)}{4}+O\left( e^{-\frac{2}{\e}}\right),\\
\kappa_1(\e)&\left\{
\begin{array}{ll}
=-\frac{(p-1)(5-p)}{4}+O\left(e^{-\frac{3-p}{\e}}\right)& (1<p<3),\\
\le -1+O\left(e^{-\frac{2}{\e}}\right) & (3\le p).
\end{array}
\right.
\end{split}
\]
In particular, if $\e>0$ is small, then $\kappa_1(\e)<0<\kappa_0(\e)$ and $w(x;\e)$ is nondegenerate.
\end{proposition}
Using (\ref{S71P1E0}), we immediately have
\begin{corollary}\label{E_e}
Let $w(x;\e)$ be the increasing solution. Then, $w$ satisfies $(\e w')^2-F(w)=-F(w(0;\e))$ for $x\in (0,1)$, and $-F(w(0;\e))\rightarrow 0$ ($\e\downarrow 0$).
Moreover, $F(w(\,\cdot\,;\e))\xrightarrow[]{\e\downarrow 0}0$ pointwisely in $[0,1]$.
\end{corollary}
We recall some known results of the \lq\lq limiting" operator $L^*:=\frac{d^2}{d s^2}+f'(w^*(s))$ on $\R_+$ with the Neumann boundary condition.
The operator $L^*$ has a continuous spectrum $(-\infty,-1]$ and may have discrete eigenvalues outside $(-\infty,-1]$ (\cite[p. 140]{H81}).
We study the nonnegative eigenvalues of $L^*$.
\begin{proposition}
The eigenvalue problem
\begin{equation}\label{S71P3Eq0}
L^*\phi^*=\kappa^*\phi^*\ \ \textrm{in}\ \ \R_+,\qquad(\phi^*)'(0)=0,\qquad{\phi^*}\in H^1(\R_+)
\end{equation}
has the unique nonnegative eigenvalue. Moreover, this eigenvalue is the first one which is simple, $\kappa^*_0:=(p-1)(p+3)/4$, and the associated eigenfunction is $\phi^*_0(s):=(w^*)^{(p+1)/2}$.
\end{proposition}

By direct calculation we see that
\[
\phi^*_1:=(w^*)^{\frac{3-p}{2}}-\frac{p+3}{2(p+1)}(w^*)^{\frac{p+1}{2}}\quad\textrm{and}\quad\kappa^*_1:=-\frac{(p-1)(5-p)}{4}
\]
satisfy
\[
L^*\phi^*_1=\kappa^*_1\phi^*_1\ \ \textrm{in}\ \ \R_+,\qquad (\phi^*)'_1(0)=0,\qquad\phi^*_1\in H^1(\R_+)\ (1<p<3).
\]
It is known that if $1<p<3$, then $\kappa^*_1$ is the second eigenvalue and that if $p\ge 3$, then $L^*$ has only one eigenvalue above $-1$.
In particular, $0$ is not an eigenvalue and $L^*$ is invertible.
\begin{corollary}\label{kappa_1}
For each $p>1$, there is $\delta>0$ such that $L^*$ has no eigenvalue in $[-\delta,\infty)$ except $\kappa^*_0$.
\end{corollary}

\subsection{Notation}
{}From now on we assume $q\in (\max\{N/2,2\},N)$.
Under this assumption the inclusion $W^{2,q}(B)\hookrightarrow C^0(B)$ is continuous ($q>N/2$) and $1/r\in L^q(B)$ ($q<N$).
Let $\calR$ denote the set of the radial functions on $B$.
Let $\bbX:=\{u\in W^{2,q}(B)\cap\calR;\ \frac{\partial}{\partial\nu}u=0\ \textrm{on}\ \partial B\}$ and $\bbY:=L^q(B)\cap\calR$.
By $\calB(\bbX,\bbY)$ we denote the space of the bounded operators from $\bbX$ to $\bbY$ equipped with the operator norm $\left\|\,\cdot\,\right\|_{\calB(\bbX,\bbY)}$.
In this section we work on $\bbX$ and $\bbY$ and define $f(U):=-U+U|U|^{p-1}$ in order that $f$ is defined in $\R$.
The regular radial solutions of (\ref{N}) satisfy
\begin{equation}\label{R}
\begin{cases}
\e^2\left(U''+\frac{N-1}{r}U'\right)+f(U)=0, & 0<r<1,\\
U'(0)=U'(1)=0.
\end{cases}
\end{equation}
Let $w(r)$ be the increasing solution given in Proposition~\ref{S71P1}.
We find a solution of (\ref{R}) near $w(r)$.
Substituting $U(r)=v(r)+w(r)$ into (\ref{R}), we obtain the following equation of the error term $v$:
\begin{equation}\label{EE}
\calL v+\calM w+\calN (v,w)=0,
\end{equation}
where
\begin{align*}
\calL v&:=\e^2\left(v''+\frac{N-1}{r}v'\right)+f'(w)v,\\
\calM w&:=\e^2\frac{N-1}{r}w',\\
\calN (v,w)&:=f(v+w)-f(w)-f'(w)v.
\end{align*}
$\calL$ is an approximation of the linearized operator $\e^2\left(\frac{d^2}{dr^2}+\frac{N-1}{r}\frac{d}{dr}\right)+f'(U)$.
We easily see that $\calL\in\calB (\bbX,\bbY)$ and that $\calN(\,\cdot\,,w)$ is a nonlinear mapping from $\bbX$ to $\bbY$, since the inclusion $\bbX\hookrightarrow C^0(B)$ is continuous.

\subsection{A priori estimates}
We use the following a priori estimate in order to use the dominated convergence theorem.
\begin{proposition}\label{S73P1}
Let $\Omega$ be a bounded domain. Let $\phi\in C^0(\Omega)\cap H^1(\Omega)$. If
\[
\e^2\Delta\phi-V(x)\phi=0\ \ \textrm{in}\ \ \Omega,
\]
and if $V(x)>C_0>0$ in $\Omega$, then there is $C_1>0$ independent of $\e$ such that
\[
|\phi(x)|\le 2\left\|\phi\right\|_{\infty}\exp\left(-\frac{C_1{\rm dist} (x,\partial\Omega)}{\e}\right),
\]
where ${\rm dist}(x,\partial\Omega)$ is the distance from $x$ to $\partial\Omega$.
\end{proposition}
When $\phi\in C^2(\Omega)$, Proposition~\ref{S73P1} was proved in \cite[Lemma~4.2]{F73} for a general operator. The proof of \cite[Lemma~4.2]{F73} works in the generalized sense, hence Proposition~\ref{S73P1} holds. This proposition was also used in \cite{NT91}.

\begin{lemma}\label{AE1}
Let $(\kappa_{\e},\phi)\in \R\times\bbX$ ($\|\phi\|_{\infty}=1$) be the eigenpair of
\[
\begin{cases}
\calL\phi=\kappa_{\e} \phi & \textrm{in}\ \ B,\\
\frac{\partial}{\partial\nu}\phi=0 & \textrm{on}\ \ \partial B.
\end{cases}
\]
If there is a small $\delta_0>0$ such that $|\kappa_{\e}|<\delta_0$ for small $\e>0$, then there are $C_0>0$ and $C_1>0$ independent of $\e>0$ such that
\[
|\phi(r)|\le C_0e^{-C_1(1-r)/\e}\ \ (0\le r\le 1).
\]
\end{lemma}

\begin{proof}
We see by Proposition~\ref{S71P1} that there are $C_2>0$ and $\delta_1>0$ such that $f'(w)<-C_2<0$ $(0\le r\le 1-\delta_1\e)$.
The function $\phi$ satisfies
\[
\e^2\left(\phi''+\frac{N-1}{r}\phi'\right)+f'(w)\phi-\kappa_{\e}\phi=0.
\]
Since $\delta_0>0$ is small, there is $C_3>0$ such that $f'(w)-\kappa_{\e}<f'(w)+\delta_0<-C_3<0$ ($0\le r\le 1-\delta_1\e$). It follows from Proposition~\ref{S73P1} that there are $C_4>0$ and $C_5>0$ such that $\textrm{dist}(r,\partial B)=1-r-\delta_1\e$ and 
\[
|\phi(r)|\le C_4 e^{-C_5(1-r-\e \delta_1)/\e}=C_4e^{\delta_1C_5}e^{-C_5(1-r)/\e}\ \ (0\le r\le 1-\delta_1\e).
\]
If we take $C_4$ large enough, then this inequality holds for $r\in [0,1]$, because $\left\| \phi\right\|_{\infty}=1<C_4\le C_4 e^{-C_5(1-r-\e \delta_1)/\e}$ ($1-\delta_1\e<r<1$).
The proof is complete.
\end{proof}

\subsection{Invertibility of $\calL$}
The main technical lemma in Section~7 is the invertibility of $\calL$.
\begin{lemma}\label{InvA}
The operator $\calL\in\calB(\bbX,\bbY)$ is invertible, and there is $C>0$ independent of small $\e>0$ such that $\|\calL^{-1}\|_{\calB(\bbY,\bbX)}\le C$.
\end{lemma}

\begin{proof}
It is enough to show that there is a small $\delta>0$ such that $\calL$ has no eigenvalue in $[-\delta,\delta]$.
We show this by contradiction.

Suppose the contrary, i.e., there is a small $\delta>0$ such that $\calL$ has an eigenvalue $\kappa_{\e}$ in $[-\delta,\delta]$ for small $\e>0$.
Then there is an eigenfunction $\phi$ such that
\[
\begin{cases}
\e^2\Delta\phi+f'(w)\phi=\kappa_{\e}\phi & \textrm{in}\ \ B,\\
\frac{\partial}{\partial\nu}\phi=0 & \textrm{on}\ \ \partial B.
\end{cases}
\]
Without loss of generality we can assume $\left\|\phi\right\|_{\infty}=1$.
Because of Proposition~\ref{S71P1}, $f'(w)-\kappa_{\e}<-\delta_0<0$ $(0\le r\le 1-\delta_1\e)$.
By Lemma~\ref{AE1} we have
\[
|\phi(r)|\le C_0e^{-C_1(1-r)/\e}\ \ (0\le r\le 1).
\]
We let $\tilde{\phi}(s):=\phi(r)$, $\tilde{w}(s):=w(r)$, and $r=1-\e s$.
Moreover, we set $\bar{\phi}(s):=\tilde{\phi}(|s|)$ and $\bar{w}(s):=\tilde{w}(|s|)$.
Then $\bar{\phi}$ satisfies
\begin{equation}\label{InvAEq1}
\begin{cases}
\bar{\phi}''+a(s)\bar{\phi}'+f'(\bar{w}(s))\bar{\phi}=0, & -\frac{1}{\e}<s<\frac{1}{\e},\\
\bar{\phi}'(0)=\bar{\phi}'(\pm\frac{1}{\e})=0,\\
\left\|\bar{\phi}\right\|_{\infty}=1,
\end{cases}
\end{equation}
where
\[
a(s):=-\textrm{sign}(s)\frac{(N-1)\e^2}{1-\e|s|}.
\]
Let $R>0$.
If $\e>0$ is small, then $2R<1/\e$, $a(s)\in L^{\infty}(-2R,2R)$.
By \cite[Theorem~9.11]{GT83} we see that $\|\bar{\phi}\|_{H^{2}(-R,R)}\le C_2$.
Since the inclusion $H^{2}(-R,R)\hookrightarrow C^{1,\gamma}(-R,R)$ ($0<\gamma<1/2$) is continuous, $\|\bar{\phi}\|_{C^{1,\gamma}(-R,R)}\le C_3$.
By the Arzel\`a-Ascoli theorem we can choose a subsequence $\{\bar{\phi}\}_{\e>0}$ such that the following holds: 
There is $\bar{\phi}_R\in C^1(-R,R)$ such that $\bar{\phi}\xrightarrow[]{\e\downarrow 0}\bar{\phi}_R$ in $C^1(-R,R)$.
It follows from (\ref{InvAEq1}) that $\bar{\phi}''$ converges in $C^0(-R,R)$ as $\e\downarrow 0$.
Since the operator $\frac{d^2}{ds^2}$ that is defined on $\{ u\in C^2(-R,R);\ u'(0)=0\}$ is a closed operator in $C^1(-R,R)$, $\bar{\phi}_R\in C^2$ and $\bar{\phi}''\xrightarrow[]{\e\downarrow 0}\bar{\phi}''_R$ in $C^0(-R,R)$.
Let $\{ R_j\}_{j\ge 1}$ ($0<R_1<R_2<\cdots\rightarrow\infty$) be a sequence diverging to $\infty$.
For each fixed $R_j$, there is $\e>0$ such that $[0,R_j)\subset(-1/\e,1/\e)$.
Using the expanding domains $\{[0,R_j)\}_{j\ge 1}$ and a diagonal argument, we see that there is $\bar{\phi}^*\in C^2(\R)$ and $\kappa^*\in [-\delta,\delta]$ such that $\bar{\phi}\xrightarrow[]{\e\downarrow 0}\bar{\phi}^*$ in $C^2_{loc}(\R)$ and $\kappa_{\e}\rightarrow\kappa^*$.
We set $\phi^*(s):=\bar{\phi}^*(s)$ ($s\in\overline{\R}_+$).
Since, for each $\varphi\in C_0^1(\overline{\R}_+)$,
\[
\int_0^{1/\e}\left(\tilde{\phi}''-\frac{(N-1)\e^2}{1-\e s}\tilde{\phi}'+f'(\tilde{w})\tilde{\phi}-\kappa_{\e}\tilde{\phi}\right)\varphi ds=0,
\]
we have
\[
\int_{\overline{\R}_+}\left( (\phi^*)''+f'(w^*)\phi^*-\kappa^*\phi^*\right)\varphi dx=0,\ \ \left\|\phi^*\right\|_{\infty}=1,\ \ (\phi^*)'(0)=0.
\]
Since $\left\|\phi^*\right\|_{\infty}=1$, $\phi^*\not\equiv 0$.

We will show that $\phi^*\in H^1(\R_+)$. Multiplying $\calL\phi=\kappa_{\e}\phi$ by $r^{N-1}\phi$ and integrating it, we have
\[
\int_0^1\left(\e^2(\phi')^2-f'(w)\phi^2\right)r^{N-1}dr=-\int_0^1\kappa_{\e}\phi^2r^{N-1}dr.
\]
Making the change of variables $r=1-\e s$, we have
\[
\int_0^{1/\e}\left((\tilde{\phi}')^2+\tilde{\phi}^2\right)(1-\e s)^{N-1}ds=
\int_0^{1/\e}\left( p\tilde{w}^{p-1}-\kappa_{\e}\right)\tilde{\phi}^2(1-\e s)^{N-1}ds.
\]
By Lemma~\ref{AE1} we have $|\tilde{\phi}(s)|\le C_0e^{-C_1s}$. The right-hand side is bounded uniformly in $\e$, hence
\[
\left|\int_0^{1/\e}\left( p\tilde{w}^{p-1}-\kappa_{\e}\right)\tilde{\phi}^2(1-\e s)^{N-1}ds\right|\le C_2.
\]
Using Fatou's lemma, we have
\begin{multline*}
\int_{\R_+}\left\{ ((\phi^*)')^2+(\phi^*)^2\right\}ds
=\int_{\R_+}\liminf_{\e\downarrow 0}\chi_{[0,1/\e]}\left((\tilde{\phi}')^2+\tilde{\phi}^2\right)(1-\e s)^{N-1}ds\\
\le\liminf_{\e\downarrow 0}\int_0^{1/\e}\left((\tilde{\phi}')^2+\tilde{\phi}^2\right)(1-\e s)^{N-1}ds\le C_2,
\end{multline*}
where $\chi_{[0,1/\e]}$ is the characteristic function. Thus $\phi^*\in H^1(\R_+)$. Therefore $(\kappa^*,\phi^*)$ is an eigenpair of (\ref{S71P3Eq0}). Since (\ref{S71P3Eq0}) has no eigenvalue in $[-\delta,\delta]$ (Corollary~\ref{kappa_1}), we obtain a contradiction.
\end{proof}

\subsection{Contraction mapping}
Because of Lemma~\ref{InvA}, $\calL$ is invertible provided that $\e>0$ is small.
The equation (\ref{EE}) can be transformed to
\[
v=\calT (v),
\]
where
\[
\calT (v):=-\calL^{-1}[\calM w]-\calL^{-1}[\calN (v,w)].
\]
In order to show that $\calT$ is a contraction mapping on a small ball in $\bbX$ we show that $\left\|\calL^{-1}[\calM w]\right\|_{\bbX}$ is small.
\begin{lemma}\label{M}
$\left\| \calL^{-1}[\calM w]\right\|_{\bbX}=o(\e)$.
\end{lemma}

\begin{proof}
By Lemma~\ref{InvA} we have
\[
\left\|\calL^{-1}[\calM w]\right\|_{\bbX}\le\left\|\calL^{-1}\right\|_{\calB(\bbX,\bbY)}\left\| \calM w\right\|_{\bbY}\le C\left\|\calM w\right\|_{\bbY}.
\]
We show that $\left\|\calM w\right\|_{\bbY}=o(\e)$.
Using Corollary~\ref{E_e}, we have
\begin{align*}
\left\|\calM w\right\|_{\bbY}^q &=\left\|\e^2\frac{N-1}{r}w'\right\|_q^q\\
&=(N-1)^q\e^q\int_0^1(\e w')^qr^{N-q-1}dr\\
&=(N-1)^q\e^q\int_0^1\left( F(w(r))-F(w(0))\right)^{\frac{q}{2}}r^{N-q-1}dr.
\end{align*}
Since $|F(w(r))-F(w(0))|\le C$ $(0\le r\le 1)$ and $|F(w(r))-F(w(0))|\xrightarrow{\e\downarrow 0}0$ pointwisely in $[0,1]$, we see that $\left(F(w(r))-F(w(0))\right)^{q/2}r^{N-q-1}\xrightarrow{\e\downarrow 0}0$ pointwisely in $(0,1]$ and that 
\[
|\left(F(w(r))-F(w(0))\right)^{q/2}r^{N-q-1}|\le C^{q/2}r^{N-q-1}\in L^1(0,1).
\]
Here, we use the condition $q<N$.
The dominated convergence theorem says that
\[
\int_0^1\left( F(w(r))-F(w(0))\right)^{\frac{q}{2}}r^{N-q-1}dr\rightarrow 0\quad (\e\downarrow 0).
\]
Since $\left\|\calM w\right\|_{\bbY}^q=\e^q\cdot o(1)=o(\e^q)$, the proof is complete.
\end{proof}

Let $\bbB_{\e}:=\{ u\in\bbX;\ \left\|u\right\|_{\bbX}<\e\}$.
We show that the Lipschitz constant of $\calN (\,\cdot\,,w)$ in $\bbB_{\e}$ is small.
\begin{lemma}\label{NN}
Let $\e>0$ be small. If $v_1,v_2\in\bbB_{\e}$, then
\[
\left\|\calL^{-1}[\calN (v_1,w)]-\calL^{-1}[\calN (v_2,w)]\right\|_{\bbX}\le o(1)\left\| v_1-v_2\right\|_{\bbX}\quad (\e\downarrow 0).
\]
In particular, the Lipschitz constant of $\calL^{-1}[\calN (\,\cdot\,,w)]:\bbB_{\e}\longrightarrow\bbX$ is less than one.
\end{lemma}

\begin{proof}
By Lemma~\ref{InvA} we have
\begin{align*}
\left\| \calL^{-1}[\calN (v_1,w)]-\calL^{-1}[\calN(v_2,w)]\right\|_{\bbX}&\le\left\| \calL^{-1}\right\|_{\calB(\bbX,\bbY)}\left\|\calN (v_1,w)-\calN(v_2,w)\right\|_{\bbY}\\
&\le C\left\| \calN(v_1,w)-\calN (v_2,w)\right\|_{\bbY}.
\end{align*}
We show that $\left\|\calN (v_1,w)-\calN (v_2,w)\right\|_{\bbY}\le o(\e)\left\| v_1-v_2\right\|_{\bbX}$ for $v_1,v_2\in\bbB_{\e}$.
We have
\begin{align*}
\left\|\calN (v_1,w)-\calN (v_2,w)\right\|_q & =\left\| f(v_1+w)-f(v_2+w)-f'(w)(v_1-v_2)\right\|_q\\
&=\left\| f'(\tilde{w})(v_1-v_2)-f'(w)(v_1-v_2)\right\|_q\\
&\le\left\| f'(\tilde{w})-f'(w)\right\|_{\infty}\left\| v_1-v_2\right\|_q.
\end{align*}
Here, $\tilde{w}(x)$ is a function such that
\[
\min\{ v_1(x)+w(x),v_2(x)+w(x)\}\le\tilde{w}(x)\le\max\{ v_1(x)+w(x),v_2(x)+w(x)\},
\]
hence $\left\| w(x)-\tilde{w}(x)\right\|_{\infty}\le C\e$.
We have $\left\| f'(\tilde{w})-f'(w)\right\|_{\infty}=o(1)$ $(\e\downarrow 0)$.
Thus,
\begin{align*}
\left\|\calN (v_1,w)-\calN (v_2,w)\right\|_q & \le o(1)\left\| v_1-v_2\right\|_q\\
& \le o(1) C\left\| v_1-v_2\right\|_{\bbX}.
\end{align*}
The proof is complete.
\end{proof}

Using Lemma~\ref{NN}, we will show that $\calT$ is a contraction mapping in $\bbB_{\e}$.
\begin{lemma}\label{T}
Let $\e>0$ be small.
Then $T(\bbB_{\e})\subset\bbB_{\e}$.
Moreover, the Lipschitz constant of $\calT :\bbB_{\e}\longrightarrow\bbB_{\e}$ is less than one, i.e., there is $\delta\in (0,1)$ such that if $v_1,v_2\in\bbB_{\e}$, then $\left\|\calT (v_1)-\calT (v_2)\right\|_{\bbX}\le\delta\left\|v_1-v_2\right\|_{\bbX}$.
\end{lemma}

\begin{proof}
First, we will show that $T(\bbB_{\e})\subset\bbB_{\e}$.
By Lemma~\ref{NN} we have that, for $v\in\bbB_{\e}$,
\begin{align}
\left\|\calL^{-1}[\calN( v,w)]\right\|_{\bbX}&=\left\| \calL^{-1}[\calN (v,w)]-\calL^{-1}[\calN (0,w)]\right\|_{\bbX}\nonumber\\
&\le o(1)\left\| v-0\right\|_{\bbX},\label{TEq0}
\end{align}
where we use Lemma~\ref{NN}.
We have that, for $v\in\bbB_{\e}$,
\begin{align*}
\left\| \calT (v)\right\|_{\bbX} &=\left\|\calL^{-1}[\calM w]+\calL^{-1}[\calN (v,w)]\right\|_{\bbX}\\
&\le\left\|\calL^{-1}[\calM w]\right\|_{\bbX}+\left\|\calL^{-1}[\calN (v,w)]\right\|_{\bbX}\\
&\le o(\e)+o(1)\left\| v\right\|_{\bbX}=o(\e),
\end{align*}
where we use Lemma~\ref{M} and (\ref{TEq0}).
Thus, if $\e>0$ is small, then, for all $v\in\bbB_{\e}$, $\calT (v)\in\bbB_{\e}$.

Second, we show that the Lipschitz constant of $\calT(\,\cdot\,)$ in $\bbB_{\e}$ is $o(1)$.
We have that, for $v_1,v_2\in\bbB_{\e}$,
\begin{align*}
\left\|\calT (v_1)-\calT (v_2)\right\|_{\bbX}
&=\left\|\calL^{-1}[\calN (v_1,w)]-\calL^{-1}[\calN (v_2,w)]\right\|_{\bbX}\\
&\le o(1)\left\| v_1-v_2\right\|_{\bbX},
\end{align*}
where we use Lemma~\ref{NN}.
Then the later part of the lemma holds.
\end{proof}

Applying the contraction mapping theorem to $\calT$ which is defined in $\bbB_{\e}$, we obtain the following:
\begin{corollary}\label{Branch}
There is a large $\lambda_0>0$ such that (\ref{R}) has a one-parameter family of positive solutions $\{(\lambda,U(r,\lambda))\}_{\lambda>\lambda_0}$ and $\left\| U-w\right\|_{\infty}<C\e(=C/\sqrt{\lambda})$.
\end{corollary}

\begin{proof}
Since the condition $\| U-w\|_{\infty}<C\e$ does not guarantee the positivity of $U$, we have to check the positivity of $U$.
We see by a priori estimate that $U$ is a classical solution.
If there is $r_0\in [0,1]$ such that $U(r_0)<0$, then there is $r_1\in [0,1]$ such that $\min_{0\le r<1}U(r)=U(r_1)<0$.
Then, $0\le \e^2U''(r_1)=-f(U(r_1))$, hence $U(r_1)\le -1$.
We obtain a contradiction, because $\|U-w\|_{\infty}<C\e$.
Therefore $U\ge 0$.
The equality does not hold, because of the strong maximum principle.
Thus, $U>0$.
\end{proof}

\subsection{Nondegeneracy}
\begin{lemma}\label{ND}
Let $\{(\lambda,U)\}$ be a family of solutions obtained in Corollary~\ref{Branch}, and let $L:=\e^2\left(\frac{d^2}{dr^2}+\frac{N-1}{r}\frac{d}{dr}\right)+f'(U)$.
Then there is small $\delta>0$ such that $L$ has no eigenvalue in $[-\delta,\delta]$ if $\e>0$ is small.
In particular, $U$ is nondegenerate in $\bbX$.
\end{lemma}

\begin{proof}
Let $\delta>0$ be a constant used in the proof of Lemma~\ref{InvA}.
Then, $\calL$ has no eigenvalue in $[-\delta,\delta]$.
We show that $L$ has no eigenvalue in $[-\delta/2,\delta/2]$.
Suppose the contrary, i.e., there is $\kappa_{\lambda}\in [-\delta/2,\delta/2]$ such that $L\phi=\kappa_{\lambda}\phi$ $(\|\phi\|_{\bbX}=1)$.
Then,
\begin{equation}\label{NDEq0}
(\calL-\kappa_{\lambda})\phi+(f'(U)-f'(w))\phi=0.
\end{equation}
Because of Lemma~\ref{InvA}, $(\calL-\kappa_{\lambda})$ is invertible and there is $C>0$ independent of small $\e>0$ such that $\left\| (\calL-\kappa_{\lambda})^{-1}\right\|_{\calB (\bbX,\bbY)}\le C$.
It follows from Corollary~\ref{Branch} that $\left\| f'(U)-f'(w)\right\|_{\infty}\rightarrow 0$ as $\e\downarrow 0$.
By (\ref{NDEq0}) we have
\begin{align*}
\left\|\phi\right\|_{\bbX} &=\left\| (\calL -\kappa_{\lambda})^{-1}\left[\left( f'(U)-f'(w)\right)\phi\right]\right\|_{\bbX}\\
&\le\left\| (\calL-\kappa_{\lambda})^{-1}\right\|_{\calB(\bbX,\bbY)}\left\| \left( f'(U)-f'(w)\right)\phi\right\|_{\bbY}\\
&\le C\left\| f'(U)-f'(w)\right\|_{\infty}\left\|\phi\right\|_{\bbY}\rightarrow 0\quad (\e\downarrow 0),
\end{align*}
which contradicts that $\|\phi\|_{\bbX}=1$.
\end{proof}

\subsection{Asymptotic behavior of $\lambda_1(\gamma)$ $(\gamma\downarrow 0)$}
\begin{theorem}\label{S77C1}
Let $(\lambda_1(\gamma),U_1(r,\gamma))\in\calC_1$.
Then, $U_1(r,\gamma)$ is a boundary concentrating solution obtained in Corollary~\ref{Branch} if $\gamma>0$ is small.
\end{theorem}

\begin{proof}
Since $|w(r)-w(r)^p|\le C$ $(0\le r\le 1)$ and $w(r)\xrightarrow{\e\downarrow 0}0$ pointwisely in $[0,1)$, by the dominated convergence theorem we see that if $r>1/2$, then
\begin{align*}
\e^2\frac{|w'(r)|}{r}&\le\frac{1}{r}\int_0^r\left(|w(s)|+|w(s)|^p\right)ds\\
&\le 2\int_0^1\left(|w(s)|+|w(s)|^p\right)ds\rightarrow 0\quad (\e\downarrow 0).
\end{align*}
Hence, $\left\|\calM w\right\|_{\infty}\rightarrow 0$ $(\e\downarrow 0)$.
Since $v$ satisfies (\ref{EE}), by the elliptic regularity we see that $\left\| v\right\|_{C^1}\rightarrow 0$ $(\e\downarrow 0)$.
If $v(x_0)+w(x_0)=1$, then $x_0>1/2$ and by phase plane argument we see that $w'(x_0)(>0)$ is large. Since $v'(x_0)$ is small, $v'(x_0)+w'(x_0)>0$.
Therefore, $\calZ_{[0,1]}[v(\,\cdot\,)+w(\,\cdot\,)-1]=1$ provided that $\e>0$ is small.
The boundary concentrating solution obtained by Corollary~\ref{Branch} belongs to $\calC_1$, because of (\ref{Th3E0}).
The proof of (\ref{Th3E0}) is postponed until the proof of Theorem~\ref{Th2} in Section~8.
\end{proof}

\begin{corollary}\label{S77T1}
Let $(\lambda_1(\gamma),U_1(r,\gamma))\in\calC_1$.
Then, $\lambda_1(\gamma)\rightarrow\infty$ $(\gamma\downarrow 0)$.
Moreover, $\frac{d\lambda_1(\gamma)}{d\gamma}<0$ for small $\gamma>0$.
\end{corollary}

\begin{proof}
Because of Lemma~\ref{ND}, if $\e>0$ is small, then $U$ is nondegenerate, which indicates the following:
If $\lambda_1(\gamma)>0$ is large, then
\begin{equation}\label{S77T1E1}
\frac{d\lambda_1(\gamma)}{d\gamma}\neq 0.
\end{equation}
Because of Corollary~\ref{Branch}, $\left\| U-w\right\|_{\infty}<C\e$.
Since $\e=(\lambda_1(\gamma))^{-\frac{1}{2}}$, we have
\begin{align*}
(\gamma(\lambda_1)=)U(0) &\le|U(0)-w(0)|+|w(0)|\\
&<O(\frac{1}{\sqrt{\lambda_1}})+o(1)\quad (\lambda_1\rightarrow\infty).
\end{align*}
Therefore,
\begin{equation}\label{S77T1E2}
\gamma(\lambda_1)\rightarrow 0\quad (\lambda_1\rightarrow\infty).
\end{equation}
By (\ref{S77T1E1}) and (\ref{S77T1E2}) we see that
\begin{equation}\label{S77T1E3}
\frac{d\lambda_1(\gamma)}{d\gamma}<0\ \ \textrm{if}\ \ \lambda_1(\gamma)>0\ \ \textrm{is large.}
\end{equation}
By (\ref{S77T1E2}) and (\ref{S77T1E3}) we see that $\lambda_1(\gamma)\rightarrow\infty$ $(\gamma\downarrow 0)$.
\end{proof}

\section{Proof of Theorem~\ref{Th2}}
\begin{proof}[Proof of Theorem~\ref{Th2}]
(i) is mentioned in Section~1.
(ii) follows from Theorem~\ref{S5T2}.
(iii) is proved in Theorem~\ref{S6T1}.
(v) follows from Theorem~\ref{S77C1} and Corollary~\ref{S77T1}.
Let $u(s,\gamma)$ be the solution of (\ref{S3E2}).
Then, we already saw in Lemmas~\ref{BDDB} and \ref{GPR2} that if $\gamma\neq 1$, then $u(s,\gamma)$ has infinitely many critical points.
Let $\{s_n\}_{n=1}^{\infty}$ $(0<s_1<s_2<\cdots)$ denote the set of the critical points of $u(s,\gamma)$.
Since $\lambda_n(\gamma)=s_n^2$, (vi) holds.
Corollary~\ref{S77T1} says that $\lambda_1(\gamma)\rightarrow\infty$ $(\gamma\downarrow 0)$. Since $\lambda_1(\gamma)<\lambda_2(\gamma)<\cdots$, (iv) holds.

We prove (\ref{Th3E0}).
Let $(\lambda_0,U_0)\in\calS$, and let $m:=\calZ_{[0,1]}[U(\,\cdot\,)-1]$.
Proposition~\ref{LPR} says that there is a one-parameter family of solutions $\calC:=(\lambda(\gamma),U(\gamma))$ such that $(\lambda(0),U(0))=(\lambda_0,U_0)$.
This branch $\calC$ can be extended to $\gamma=1$, because of Lemmas~\ref{BDDB} and \ref{GPR2}.
Note that $\calZ_{[0,1]}[U(\,\cdot\,,\gamma)-1]=m$ if $(\lambda,U)\in\calC$ and $U\not\equiv 1$.
It follows from the uniqueness of the branch near $(\bar{\lambda}_m,1)$ that $\calC\subset\calC_m$.
Thus, $(\lambda_0,U_0)\in\calC\subset\calC_m$, and the (\ref{Th3E0}) holds.
The proof is complete.
\end{proof}

\begin{proof}[Proof of Corollary~\ref{Cor1}]
Let $\underline{\lambda}:=\inf_{\gamma\in\R_+}\lambda_1(\gamma)>0$.
Because of (ii) and (iv), $\underline{\lambda}>0$.
For each $n\ge 2$, $\lambda_n(\gamma)>\lambda_1(\gamma)(\ge\underline{\lambda})$.
This implies that the first assertion holds.
The other assertion follows from the boundedness of $\{\lambda_1(\gamma)\}_{\gamma>1}$.
\end{proof}


\noindent
{\bf Acknowledgement}
The author would like to thank the referee for the careful reading of the manuscript and many helpful comments.


\end{document}